\documentclass[final]{siamltex}
\usepackage{amsfonts}
\usepackage{bbm}
\usepackage{amssymb,color}
\usepackage{mathrsfs}
\usepackage{amsmath}
\usepackage{epsfig}
\usepackage{graphicx}
\usepackage{subfigure}
\usepackage{epstopdf}
\usepackage{float}
\usepackage{tabularx}
\newtheorem{assumption}{assumption}
\newtheorem{remark}{remark}

\title{Deep learning numerical methods for high-dimensional fully nonlinear PIDEs and coupled FBSDEs with jumps\thanks{This work was supported by grants from the National Natural Science
Foundation of China (Grant Nos. 12271367, 11771060), Science and Technology Innovation Plan Of Shanghai Science and Technology Commission (No. 20JC1414200), and sponsored by Natural Science Foundation of Shanghai, China (No. 20ZR1441200). }}

\author{Wansheng Wang\thanks{Corresponding author, Department of Mathematics, Shanghai Normal University, Shanghai, 200234, China ({\tt w.s.wang@163.com}).}
                \and Jie Wang\thanks{Department of Mathematics, Shanghai Normal University, Shanghai, 200234, China.}
                \and Jinping Li\thanks{School of Science, Hainan University, Haikou, China}
                \and Feifei Gao\thanks{Department of Mathematics, Shanghai Normal University, Shanghai, 200234, China.}
                \and Yi Fu\thanks{School of Finance and Business, Shanghai Normal University, Shanghai, 200234, China.}
                }

\begin{document}

\maketitle

\begin{abstract}
We propose a deep learning algorithm for solving high-dimensional parabolic integro-differential equations (PIDEs) and high-dimensional forward-backward stochastic differential equations with jumps (FBSDEJs), where the jump-diffusion process are derived by a Brownian motion and an independent compensated Poisson random measure. In this novel algorithm, a pair of deep neural networks for the approximations of the gradient and the integral kernel is introduced in a crucial way based on deep FBSDE method. To derive the error estimates for this deep learning algorithm, the convergence of Markovian iteration, the error bound of Euler time discretization, and the simulation error of deep learning algorithm are investigated. Two numerical examples are provided to show the efficiency of this proposed algorithm.
\end{abstract}

\begin{keywords}
 parabolic integro-differential equations, forward-backward stochastic differential equations with jumps, deep learning,  error estimates
\end{keywords}

\begin{AMS}
60H35, 65C20, 65M15, 65C30, 60H10, 65M75
\end{AMS}

\pagestyle{myheadings} \thispagestyle{plain} \markboth{W. S. WANG, J.
WANG, J. P. Li, F. F. GAO AND Y. FU}{DEEP LEARNING ALGORITHMS FOR PIDES AND FBSDEJS}

\section{Introduction} The purpose of this paper is to derive error estimates for the proposed deep learning algorithm for solving high-dimensional parabolic integro-partial differential equations (PIDEs) which can be represented by high-dimensional forward-backward stochastic differential equations with jumps (FBSDEJs), because of the generalized nonlinear Feynman-Kac formula \cite{BBP1997}.

PIDEs and FBSDEJs mathematical models have been widely employed in various applications such as stochastic optimal control \cite{PP1990,P1993,TL1994,B2006,Situ05}, mathematical finance \cite{KPQ2014,Cont03,Situ05}, and so on. The existence, uniqueness and regularity of the solution to the two classes of equations have been also examined by many researchers at about the same time (see, for example, \cite{PP1990,P1993,TL1994,B2006,Situ05,KPQ2014,R1997,Z1999,W2003}). Due to the complex solution structure, however, explicit solutions of PIDEs and FBSDEJs can seldom be found. Consequently, one usually resorts to numerical methods to solve the two kinds of equations, and a volume of work has been performed on their numerical solutions. IMEX time discretizations combined with finite difference method, finite element method, or spectral method, have been used to solve low-dimensional PIDEs (see, for example, \cite{Achdou05,Pindza14,Kadalbajoo17,Wang19,WMZ2021,Mao22}), and multistep and prediction-correction schemes have been used to low-dimensional FBSDEJs (see, for example, \cite{ZHAO16,ZHAO17,Fu2016}).

With the increase of dimensionality, the traditional grid-based numerical method is no longer suitable for high-dimensional problems, and its computational complexity will increase exponentially, resulting in the so-called ``curse of dimensionality" \cite{Bellman1957}. Therefore, the resolution of nonlinear partial differential equations (PDEs) in high dimension has always been a challenge for scientists. Recently, based on the Feyman-Kac representation of the PDEs, branch diffusion process method  and Monte Carlo method  have been studied; see, for example, \cite{Gobet05,HL2012,PNT2016,Warina18}.

In recent years, machine learning and deep learning have played a great role in many fields, such as
as image recognition, automatic driving, natural language processing and so on.
This also provides a new idea for numerical approximation of high-dimensional functions, which has attracted more and more scholars' attention, since these approximation methods can overcome the problem of ``curse of dimensionality". Still based on Feyman-Kac representation of the PDEs, some machine learning techniques (see, for example, \cite{Chan19}) have proposed to solve the high-dimensional problems. With multilevel techniques and automatic differentiation, multi-layer Picard iterative methods have been developed for handling some high-dimensional PDEs with nonlinearity (see, for example, \cite{WMAT19,MATTP18,Hutzenthler20}). Using machine learning representation of the solution, the so-called Deep Galerkin method has proposed to solve PDEs on a finite domain in \cite{JK2018}. On basis of the backward stochastic differential equation (BSDE) approach first developed in \cite{PP1990}, a neutral network method was proposed to solve high-dimensional PDEs in the pioneering papers \cite{WEM18,WA17}. The idea of this algorithm is to view the BSDE as a stochastic control problem with the gradient of the solution being the policy function, which can be approximated by a deep neural network by minimizing a global loss function. Deep learning backward dynamic programming (DBDP) methods, including DBDP1 scheme and DBDP2 scheme, in which some machine learning techniques are used to estimate simultaneously the solution and its gradient by minimizing a loss function on each time step, were proposed in \cite{HPH2020}. The DBDP1 algorithm has been extended to the case of semilinear parabolic nonlocal integro-differential equations in \cite{Castro21}. Quite recently, a new deep learning algorithm was proposed to solve fully nonlinear PDEs and nonlinear second-order backward
stochastic differential equations (2BSDE) by exploiting a connection between
PDEs and 2BSDEs \cite{Beck19}.

It is worth noting that most of the above-named approximation methods are only applicable
in the case of semilinear PIDEs or nonlinear PDEs. To the best of our knowledge, only the
papers by Gonon and Schwab \cite{Gonon21a,Gonon21b}, and Castro \cite{Castro21} are devoted to the deep learning approximations of the numerical solution of linear and semilinear PIDEs. At the moment there exists no practical algorithm for high-dimensional fully nonlinear
PIDEs in the scientific literature. Consequently, the numerical solution of high-dimensional nonlinear PIDEs remains an exceedingly difficult task and deserves further study. In this work, we propose a new
algorithm for solving fully nonlinear PIDEs and nonlinear FBSDEJs. The proposed algorithm exploits a connection between PIDEs and FBSDEJs to obtain a merged formulation of
the nonlinear PIDE and the coupled FBSDEJs, whose solution is then approximated by combining a Euler time
discretization with a Markovian iteration \cite{BZ2008} and a neural network-based deep learning procedure. The error estimates of this new FBSDE algorithm (we refer to the algorithm as FBSDE since it is based on forward-backward stochastic differential equations but not only backward stochastic differential equation) are then derived by bounding the time discretization error and deep learning error, and by showing the convergence of Markovian iteration.

The paper is organized as follows. We start by introducing the deep learning-based algorithm for FBSDEJs and related PIDEs in Section 2. In Section 3, the assumptions for theoretical analysis are made and the main error estimates are given. To prove this main results, we show the convergence of Markovian iteration, bound the time discretization error, and derive the simulation error of deep learning in Sections 4, 5, and 6, respectively. Several numerical experiments with the proposed scheme are presented in Section 7. In Section 8 we finally conclude with some remarks.

\section{Deep learning-based schemes for nonlinear PIDEs and coupled FBSDEJs} In this section, we introduce the details about deep learning-based schemes for solving coupled FBSDEJs and the associated nonlinear PIDEs. We deal with nonlinear PIDEs in three steps.
\begin{itemize}
	\item  We formulate the PIDEs as FBSDEJs.
	\item By taking ``control part'' and ``integral kernel'' as policy functions, we view FBSDEJs as a stochastic control problem.
	\item We use a deep neural network to approximate high-dimensional policy function.
\end{itemize}

\subsection{Nonlinear PIDEs and coupled FBSDEJs} Let $|\cdot|$ denote the Euclidean norm in the Euclidean space, and $C^{l,k}$ denote the set of functions $f(t,x)$ with continuous partial derivatives up to $l$ with respect to $t$ and up to $k$ with respect to $x$. Let $(\Omega,\mathcal{F}, \mathbb F,P)$, $\mathbb F=(\mathcal{F}_{t})_{0\leq t<T}$, be a stochastic basis such that $ \mathcal{F}_{0} $ contain all zero $P$-measure sets, and $ \mathcal{F}_{t^+} \triangleq \bigcap_{\epsilon > 0}\mathcal{F}_{t+\epsilon}=\mathcal{F}_{t} $. The filtration $\mathbb F$ is generated by a $d$-dimension Brownian motion (BM) $ \{ W_{t} \} _{0\leq t<T} $ and a Poisson random measure $\mu$  on $\mathbb R_+\times E$, independent of $W$. In this subsection, we establish a connection between nonlinear PIDEs and coupled FBSDEJs.

Let us consider the following nonlinear PIDEs
\begin{eqnarray}\label{eq2.1}
\left\{
\begin{aligned}
\partial_{t}u+\mathcal L u+f(t,x,u,\sigma^\mathsf{T}(t,x,u)\nabla_{x}u,B[u]) &=0, \quad \;(t,x) \in [0,T)\times \mathbb R^d,\\
u(T,x)&=g, \quad \;x\in \mathbb R^d,
\end{aligned}
\right.
\end{eqnarray}
where $d\ge 1$ and $T>0$, $g$: $\mathbb R^d \rightarrow \mathbb R$ is terminal condition, the second-order nonlocal operator $\mathcal L$ is defined as follows:
\begin{eqnarray*}
\mathcal L u&=&\frac{1}{2}{\rm Tr}(\sigma\sigma^\mathsf{T}(t,x,u)\partial_{x}^{2}u)+\left\langle b(t,x,u),\nabla_{x} u\right\rangle\\
&&+\int_{E}(u(t,x+\beta(t,x,u,e))-u(t,x)-\left\langle\nabla_{x} u,\beta(t,x,u,e)\right \rangle)\lambda(de),
\end{eqnarray*}
and $B$ is an integral operator
\begin{eqnarray*}
B[u]=\int_{E}(u(t,x+\beta(t,x,e))-u(t,x))\gamma(e)\lambda(de).
\end{eqnarray*}
Here $b(t,x,y) $: $ [0,T] \times \mathbb R^d \times \mathbb R \to \mathbb R^d$, $\sigma(t,x,y)$: $[0,T] \times \mathbb R^d \times \mathbb R \to \mathbb R^{d \times d} $, $ \beta(t,x,y,e) $: $[0,T] \times \mathbb R^d \times \mathbb R \times E \to \mathbb R^d $, and $f$: $[0,T] \times \mathbb R^d \times \mathbb R \times \mathbb R^{d} \times \mathbb R \rightarrow \mathbb R$ are deterministic and Lipschitz continuous functions of linear growth which are additionally supposed to satisfy some weak coupling or monotonicity conditions, $A^{\mathsf{T}}$ denotes the transpose of a vector or matrix $A$, $ E \triangleq \mathbb R^{d} \backslash \{0\} $ is equipped with its Borel field $\mathcal E$, with compensator $\nu(de,dt)=\lambda(de) dt$, for some measurable functions $\gamma:  E\rightarrow \mathbb R $ satisfying
\begin{eqnarray}\label{eq2.2}
\sup\limits_{e \in E}|\gamma(e)| \leq K_{\gamma},
\end{eqnarray}
and $\lambda(de)$ is assumed to be a $\sigma $-finite measure on $ (E,\mathcal E ) $ satisfying
\begin{eqnarray*}
\int_{E} (1 \land |e|^2) \lambda(de) < \infty.
\end{eqnarray*}
\\

Let $ u(t,x) \in C^{1,2}([0,T] \times \mathbb R^d) $ be the unique viscosity solution of (\ref{eq2.1}). Then by the It\^o formula, one can show that the solution $u$ admits a probabilistic representation, i.e., we have (see \cite{Ma07} and \cite{BBP1997}), 
\begin{eqnarray}\label{eq2.5}
u(t,X_{t})=Y_{t},
\end{eqnarray}
and furthermore, the following relationship holds
\begin{eqnarray}\label{eq2.6}
\left\{
\begin{aligned}
&Z_{t}=\nabla_{x}u(t,X_{t})\sigma(t,X_{t},u(t,X_t)),\\
&U_{t}=u\left(t,X_{t^{-}}+\beta(t,X_{t^{-}},u(t,X_{t^{-}}),e)\right)-u(t,X_{t^{-}})\\
&\Gamma_t=B[u(t,X_t)],
\end{aligned}
\right.
\end{eqnarray}
where the quadruplet $(X_{t},Y_{t},Z_{t},\Gamma_{t}) $ is the solution of the coupled FBSDEJs
 \begin{eqnarray}\label{eq3.15 }
 \left\{
 \begin{aligned}
 X_t &=\xi+\int_{0}^{t}b(s,X_{s},Y_{s})ds+\int_{0}^{t}\sigma(s,X_{s},Y_{s})dW_{s}+\int_{0}^{t}\int_{E}\beta(X_{s^-},Y_{s^-},e)\tilde{\mu}(de,ds),\\
 Y_{t}&=g(X_{T})+\int_{t}^{T}f(s,X_s,Y_{s},Z_{s},\Gamma_{s})ds-\int_{t}^{T}Z_{s}dW_{s}-\int_{t}^{T}\int_{E}U_{s}(e)\tilde{\mu}(de,ds),
 \end{aligned}
 \right.
  \end{eqnarray}
The quadruplet $(X_{t},Y_{t},Z_{t},\Gamma_{t}) $ are called the ``forward part'', the ``backward part'', the ``control part'' and the ``jump part'', respectively. The presence of the control part $Z_t$ is crucial to find a nonanticipative solution. The above formulas (\ref{eq2.5})-(\ref{eq2.6}) are the so-called nonlinear Feynman-Kac formulas, and such formulas indicate an interesting relationship between solutions of FBSDEJs and PIDEs.


Note that the FBSDEJs (\ref{eq3.15 }) is coupled since the coefficients $b,~\sigma$ and $\beta$ depend on $Y_t$. When the coefficients $b,~\sigma$ and $\beta$ are independent of $Y_t$, the FBSDEJs (\ref{eq3.15 }) is called decoupled and can be solved in sequence. Using nonlinear Feynman-Kac formulas (\ref{eq2.5})-(\ref{eq2.6}) and the decoupled FBSDEJs, the deep learning algorithm DBDP1 proposed in \cite{HPH2020} has been extended to the semilinear PIDEs in \cite{Castro21}.

%

\subsection{Deep neural network (DNN)} In this subsection, we give a brief introduction about Deep neural networks (DNN). DNN provides effective method to solve high-dimensional approximation problems, and it is a combination of simple functions. In the past decades, there exist several type of neutral network, including Deep feedfoward neutral network, convolutional neural network (CNN) and the recurrent neural network (RNN), et.al. Deep feedforward neutral network is the simplest neural network, but it is sufficient for most PDE problems. Since it is a class of universal neural network, we consider Deep feedforward neural network in this paper.

 Let $m_{\ell}(\ell=0,\ldots,L)$ be the number of neurous in the $\ell$th layers, $L$ is layer of neural network. The first layer is the input layer, the last layer is the output layer, and another layers are the hidden layers. A feedforward neural network can be defined as the composition
 \begin{eqnarray}\label{eq2.7}
 x \in \mathbb R^{d}\to \mathcal{N}_{L}   \circ \mathcal{N}_{L-1}\circ... \circ\mathcal{N}_{1}(x)\in \mathbb R^{d_{1}},
 \end{eqnarray}
where $d$ is the dimension of $x$ and the output dimension $ d_{1}=k, k \in \mathbb R_+$. We fix $d_{1}=1$ in this paper, and
\begin{eqnarray*}\label{eq2.8}
\left\{
\begin{aligned}
&\mathcal{N}_{0}(x)=x\in \mathbb R^{d},\\
&\mathcal{N}_{\ell}(x)=\varrho(\mathbf{w}_{\ell}\mathcal{N}_{\ell-1}(x)+\mathbf{b}_{\ell})\in \mathbb R^{m_{\ell}}, \quad {\hbox{for}} \quad 1\leq \ell \leq L-1,\\
&\mathcal{N}_{L}(x)=\mathbf{w}_{L}\mathcal{N}_{L-1}(x)+\mathbf{b}_L \in \mathbb R,
\end{aligned}
\right.
\end{eqnarray*}
where $ \mathbf{w}_{\ell}\in \mathbb R^{m_{\ell}\times m_{\ell-1}}$ and $\mathbf{b}_{\ell} \in \mathbb R^{m_{\ell}}$ denote the weight matrix and bias vector, respectively, $\varrho$ is a nonlinear activation function such as the logistic sigmoid function, the hyperbolic tangent ($\tanh$) function, the rectified linear unit (ReLU) function and other similar functions. We use the ReLU function for all the hidden layers in this paper. The final layers $\mathcal{N}_{L}(x) $ is typically linear.

 Let $\theta$ denote the parameters of the neural network:
 \begin{eqnarray*}
 	\theta:=\{\mathbf{w}^{\ell}, \mathbf{b}^{\ell}\}, \qquad{\ell=1,...,L}.
 \end{eqnarray*}
 The DNN is trained by optimizing over the parameters $\theta$ by (\ref{eq2.7}).
\subsection{Time discretization of the coupled FBSDEJs} We first need to discretize equation (\ref{eq3.15 }). We consider a partition of the time interval $[0,T]$:
\begin{eqnarray*}
\tau : 0=t_0<t_1<\cdots<t_N=T, \quad  N\in \mathbb{N}
\end{eqnarray*}
with modulus $ h=\max_{n=0,1 \cdots, N} \Delta t_{n}$, $\Delta  t_{n}=t_{n+1}-t_{n} $. Then a natural time discretization of equation (\ref{eq3.15 }) is by classical Euler scheme:
\begin{align}
X^{\pi}_{t_{n+1}}=& X^{\pi}_{t_{n}}+b(t_{n},X^{\pi}_{t_{n}},Y_{t_{n}}^{\pi})\Delta  t_{n}+\sigma(t_{n},X^{\pi}_{t_{n}},Y_{t_{n}}^{\pi}) \Delta  W_{t_{n}}\nonumber\\
&+ \int_{E}\beta(t_{n},X^{\pi}_{t_{n}^{-}},Y_{t_{n}}^{\pi},e)\tilde{\mu}(de,(t_{n},t_{n+1}]),\label{eq3.1}\\
Y_{t_{n+1}}^{\pi}=&Y_{t_{n}}^{\pi}-f(t_{n},X^{\pi}_{t_{n}},Y_{t_{n+1}}^{\pi},Z_{t_{n}}^{\pi},\Gamma_{t_{n}}^{\pi})\Delta  t_{n}+Z_{t_{n}}^{\pi}\Delta  W_{t_{n}}\nonumber\\
&+\int_{E}U_{t_{n}}^{\pi}(e)\tilde{\mu}(de,(t_{n},t_{n+1}]),\label{eq3.2}
\end{align}
where $\Delta  W_{t_{n}}=W_{t_{n+1}}-W_{t_{n}}$ and $\Gamma_{t}=\int_{E}U_{t}^{\pi}(e)\gamma(e)\lambda(e)$. Note that (\ref{eq3.2}) is an explicit discretization. For the implicit discretization, which is formulated as replacing $f(t_{n},X^{\pi}_{t_{n}},Y_{t_{n+1}}^{\pi},Z_{t_{n}}^{\pi},\Gamma_{t_{n}}^{\pi})$ with $f(t_{n},X^{\pi}_{t_{n}},Y_{t_{n}}^{\pi},Z_{t_{n}}^{\pi},\Gamma_{t_{n}}^{\pi})$, the same conclusions hold as we state in Theorem 3.1 for the explicit discretization.

\subsection{Deep learning-based approximations of coupled FBSDEJs}
We already formulate the PIDEs equivalently as FBSDEJs by nonlinear Feyman-Kac formula. 
Let the quadruplet $(X_{t}^{\pi},Y_{t}^{\pi},Z_{t}^{\pi},U_{t}^{\pi}) $ be the solution of (\ref{eq3.1})-(\ref{eq3.2}) with
\begin{eqnarray}\label{eq3.4}
Y_{t}^{\pi}=u(t,X_{t}^{\pi})
\end{eqnarray}
\\and
\begin{eqnarray}\label{eq3.5}
\left\{
\begin{aligned}
&Z_{t}^{\pi}=\sigma^\mathsf{T}(t,X_{t}^{\pi},Y_{t}^{\pi})\nabla_{x}u(t,X_{t}^{\pi}),\\
&U_{t}^{\pi}(e)=u(t,X_{t^{-}}^{\pi}+\beta(t,X_{t^{-}}^{\pi},Y_{t}^{\pi},e))-u(t,X_{t^{-}}^{\pi}),
\end{aligned}
\right.
\end{eqnarray}
where $u(t,x)$ is the solution to nonlinear PIDEs \eqref{eq2.1}. We can approximate $ Z_{t}$, $U_{t} $ by a deep learning algorithm. We employ the following formulas as the policy functions:
\begin{align}\label{eq3.6}
\sigma^\mathsf{T}(t_{n},X_{t_{n}}^{\pi},Y_{t_n}^{\pi})\nabla_{x}u(t_{n},X_{t_{n}}^\pi)\in \mathbb R^{1 \times d},\quad &x \in \mathbb R^{d}, \quad n \in {0,1,\ldots,N},\\
\label{eq3.7}
u(t_{n},X_{t_{n}^{-}}^{\pi}+\beta(t,X_{t_{n}^{-}}^{\pi},Y_{t_n}^{\pi},e))-u(t_{n},X_{t_{n}^{-}}^{\pi})\in \mathbb R, \quad &x \in \mathbb R^{d}, \quad n \in {0,1,\ldots,N}.
\end{align}
More specificity, letting $\rho \in \mathbb{N}$ be the number of parameters in the neural network and $\theta \in \mathbb{R}^{\rho} $, our goal becomes finding appropriate functions $ \mathcal{A}^{\theta,\pi}: \mathbb R^d\to \mathbb{R}$, $ \mathcal{B}^{\theta,\pi}_{t_{n}}:\mathbb{R}^d\times \mathbb R\to \mathbb{R}^{d}$, $n \in \{0,1,\cdots,N-1 \}$, and $\mathcal{C}^{\theta,\pi}_{t_{n}}(e):\mathbb{R}^d\times \mathbb R \to \mathbb{R}$, $n \in \{0,1,\cdots,N-1 \}$, such that $\mathcal A^{\theta,\pi}$, $ \mathcal{B}^{\theta,\pi}_{t_{n}}$ and $\mathcal{C}^{\theta,\pi}_{t_{n}}$ can serve as good surrogates of $Y_0$, $Z_t$ and $U_t$, respectively. For all appropriate $\theta \in \mathbb{R}^{\rho} $, we define $ \mathcal{A}^{\theta,\pi}: \mathbb R^d\to \mathbb{R}$ as suitable approximation of $u(0,\xi)$:
\begin{eqnarray}\label{eq3.10}
\mathcal{A}^{\theta,\pi} \approx u(0,\xi).
\end{eqnarray}
%
and
\begin{eqnarray}\label{eq3.11}
\mathcal{B}^{\theta,\pi}_{t_{n}} \approx\sigma^\mathsf{T}(t_{n},X_{t_{n}}^{\pi})\nabla_{x}u(t_{n},X_{t_{n}}^{\pi}),
\end{eqnarray}
 \begin{eqnarray}\label{eq3.12}
 \mathcal{C}^{\theta,\pi}_{t_{n}}(e) \approx u(t_{n},X_{t_{n}^{-}}^{\pi}+\beta(t_{n},X_{t_{n}^{-}}^{\pi},e))-u(t_{n},X_{t_{n}^{-}}^{\pi}).
 \end{eqnarray}
Combining (\ref{eq3.1}), (\ref{eq3.2}), (\ref{eq3.10}), (\ref{eq3.11}) and (\ref{eq3.12}) leads to
 \begin{eqnarray}\label{eq3.16}
 \left\{
 \begin{array}{l}
 X^{\pi}_{0}=\xi,\quad Y_{0}^{\pi}=\mathcal{A}^{\theta,\pi},\vspace{2ex}\\
 \begin{aligned}
 X^{\pi}_{t_{n+1}}=& X^{\pi}_{t_{n}}+b(t_{n},X^{\pi}_{t_{n}},Y^{\pi}_{t_{n}})\Delta  t_{n}+\sigma(t_{n},X^{\pi}_{t_{n}},Y^{\pi}_{t_{n}})\Delta  W_{t_{n}} \\&+\int_{E}\beta(X^{\pi}_{t_{n}},Y^{\pi}_{t_{n}},e)\tilde{\mu}(de,(t_{n},t_{n+1}]), \vspace{2ex}
 \end{aligned}\\
 Z_{t_{n}}^{\pi}=\mathcal{B}^{\theta,\pi}_{t_n}(X^{\pi}_{t_{n}},Y^{\pi}_{t_{n}}), \qquad
 U_{t_{n}}^{\pi}(e)=\mathcal{C}^{\theta,\pi}_{t_n}(X^{\pi}_{t_{n}},Y^{\pi}_{t_{n}},e), \vspace{2ex}\\
 \begin{aligned}
 Y_{t_{n+1}}^{\pi}=&Y_{t_{n}}^{\pi}-f\left(t_{n},X^{\pi}_{t_{n}},Y_{t_{n+1}}^{\pi},Z_{t_{n}}^{\pi},\int_{E} U_{t_{n}}^{\pi}(e)\gamma(e)\lambda(de)\right)\Delta  t_{n}\\&+ Z_{t_{n}}^{\pi}\Delta  W_{t_{n}}+\int_{E}U_{t_{n}}^{\pi}(e)\tilde{\mu}(de,(t_{n},t_{n+1}]).
 \end{aligned}
 \end{array}
 \right.
 \end{eqnarray}
 Now, we set the loss function as squared approximation error
 \begin{eqnarray}\label{eq3.14 }
  \theta   \mapsto  \inf\limits_{\theta} \mathbb{E} \left[ \left| Y^{\pi}_{T} -g(X_{T}) \right|^2 \right],
 \end{eqnarray}
 associated to the terminal condition of the FBSDEJs. We then obtain the appropriate $\theta$ by minimizing the expected loss function through
 stochastic gradient descent-type algorithms (SGD).

\section{Assumptions and main results}In this paper, we will derive the a posteriori error estimates for the deep learning algorithm proposed in (\ref{eq3.16}).
 To do this, we need to introduce some notations: $\Delta x=x_{1}-x_{2}$,  $\Delta y=y_{1}-y_{2}$, $\Delta z=z_{1}-z_{2}$, $\Delta  v=v_{1}-v_{2}$,	
and make several assumptions.
\begin{assumption}
\begin{enumerate}
	\item[\rm (i)] There exist constants $k_{1}<0$, $k_{2}<0$ and $k_{g}>0$ such that
		\begin{eqnarray*}
			\begin{aligned}	
				[b(t,x_1,y)-b(t,x_2,y)]^{\mathsf{T}}\Delta x &\leq k_1 |\Delta  x|^2,\\
[f(t,x,y_1,z,v)-f(t,x,y_2,z,v)]\Delta y &\leq k_2( |\Delta  y|^2,\\
				[g(t,x_{1})-g(t,x_{2})]\Delta  x &\geq k_{g}|G\Delta  x|^2.\\ 	
			\end{aligned}
		\end{eqnarray*}
	\item[\rm (ii)] The functions $b$, $\sigma$, $\beta$, $f$ and $g$ are uniformly Lipschitz continuous with respect to $(x, y, z,v)$. In particular, there are constants $ b_{x}$, $b_{y}$, $\sigma_{x}$, $\sigma_{y}$, $\beta_{x}$, $\beta_{y}$, $f_{x}$, $f_{z}$, $f_{\Gamma}$ and $g_{x}$ such that
      \begin{eqnarray*}
      	\begin{aligned}	
      		|b (t,x_{1},y_{1})-b(t,x_{2},y_{2})|^2   \leq & b_{x}|\Delta  x|^2 +b_{y}|\Delta  y|^2,\\
      		|\sigma (t,x_{1},y_{1})-\sigma(t,x_{2},y_{2})|^2  \leq & \sigma_{x}|\Delta  x|^2 + \sigma_{y}|\Delta  y|^2,\\
      		\left|\int_{E}\beta (t,x_{1},y_{1},e)\lambda(de)-\int_{E}\beta(t,x_{2},y_{2},e)\lambda(de)\right|^2  \leq & \beta_{x}|\Delta  x|^2 + \beta_{y}|\Delta  y|^2,\\
      		|f(t,x_{1},y_{1},z_{1},v_{1})-f(t,x_{2},y_{2},z_{2},v_{2})|^2  \leq&  f_{x}|\Delta  x|^2 + f_{y}|\Delta  y|^2 \\&+f_{z}|\Delta  z|^2 + f_{\Gamma}|\Delta  v|^2,\\
      		|g(t,x_{1})-g(t,x_{2})|^2  \leq & g_{x}|\Delta  x|^2.
      	\end{aligned}
      \end{eqnarray*}
	\item[\rm (iii)] $b(t,0,0)$, $\sigma(t,0,0)$ and $\beta(t,0,0)$ are bounded. In particular, there are constants $b_{0}$, $\sigma_{0}$, $\beta_{0}$, $f_{0}$ and $g_{0}$ such that
	\begin{eqnarray*}
	\begin{aligned}	
		|b (t,x,y)|^2  & \leq b_{0} + b_{x}|x|^2 + b_{y}|y|^2,\\
		|\sigma (t,x,y)|^2 & \leq \sigma_{0} + \sigma_{x}|x|^2 + \sigma_{y}|y|^2,\\
		\left|\int_{E}\beta (t,x,y,e)\lambda(de)\right|^2 & \leq \beta_{0} +\beta_{x}|x|^2 + \beta_{y}|y|^2,\\
		|f(t,x,y,z,v)|^2  & \leq f_{0} + f_{x}|x|^2 + f_{y}|y|^2 +f_{z}|z|^2 + f_{\Gamma}|v|^2,\\
		|g(x)|^2 & \leq g_{0}+g_{x}|x|^2.
	\end{aligned}
\end{eqnarray*}
\end{enumerate}
\end{assumption}

It should be emphasized that here $b_x$ et al. are constants, not partial derivatives. For convenience, we also suppose that $\mathbf M$ is an upper bound for all these constants above.

The following assumption will be used in bounding the time discretization error.
\begin{assumption}
The coefficients $b$, $\sigma$, $\beta$, $f$ are uniformly H$\ddot{o}$lder-$\frac{1}{2}$ continuous with respect to $t$. We also assume the same constant $\mathbf M$ to be the upper bound of the square of the H\"oder constants.
\end{assumption}

Now we state an assumption which plays a key role in error analysis of numerical methods for coupled FBSDEJs problems.	
\begin{assumption}
	One of the following five cases holds:
\begin{enumerate}
\item [\rm (i)] Small time duration, that is, $T$ is small.
\item [\rm (ii)] Weak coupling of $Y$ into the forward SDEJ (\ref{eq3.15 }), that is $b_{y}$,  $\sigma_{y}$ and $\beta_{y}$ are small. In particular, if $b_{y}=\sigma_{y}=\beta_{y}=0$, then the forward equation does not depend on the backward one.
\item [\rm (iii)] Weak coupling of $X$ into the backward SDEJ (\ref{eq3.15 }), that is, $f_{x}$ and $g_{x}$ are small. In particular, if $f_{x}=g_{x}=0$, then the backward equation does not depend on the forward one and, thus, the backward SDEJ (\ref{eq3.15 }) are also decoupled.
\item [\rm (iv)] $b$ is strongly decreasing in $x$, that is, $k_1$ is very negative.
\item [\rm (iv)] $f$ is strongly decreasing in $y$, that is, $k_2$ is very negative.
	\end{enumerate}
\end{assumption}

Finally, we make an assumption on the neural network approximation functions, which makes sure the systems in (\ref{eq3.16}) is well-known.
\begin{assumption}
Then functions	$\mathcal{A}^{\theta,\pi}_{0}$, $\mathcal{B}^{\theta,\pi}_{t_{n}}$ and $\mathcal{C}^{\theta,\pi}_{t_{n}}$  are measurable with linear growth.	
\end{assumption}

It is easy to verify that neural networks with common activation functions, including ReLU and sigmoid function, satisfy this assumption.
	
Assumption 1 is usually called the Lipschitz continuity and monotonicity conditions, and Assumption 3 is called weak coupling conditions (Bender and Zhang \cite{BZ2008}). We can give more specific expression later. With these assumptions, we will prove the following main theorem in this paper.
 \begin{theorem}[Error estimates for deep learning algorithm]\label{th5.1}
Under Assumptions 1, 2, 3, and 4, there exists a constant $C$, independent of $h$, $d$, such that for sufficiently small h, for any $\epsilon >0$, we have
 		\begin{eqnarray}\label{eq3.1a}
 	\begin{aligned}	
 		\mathop{\max}\limits_{0\leq n<N}\mathop{\sup}\limits_{t_{n}\leq t \leq t_{n+1}}&(\mathbb{E}|X_{t}-X_{t_{n}}^{\pi}|^2+\mathbb{E}|Y_{t}-Y_{t_{n}}^{\pi}|^2)
 +\sum_{n=0}^{N-1}\int_{t_{n}}^{t_{n+1}}\mathbb{E}|Z_{t}-Z_{t_{n}}^{\pi}|^2dt \\ &+\sum_{n=0}^{N-1}\int_{t_{n}}^{t_{n+1}}\mathbb{E}|\Gamma_{t}-\Gamma_{t_{n}}^{\pi}|^2dt
  \leq C[h^{1-\epsilon}+\mathbb{E}|g(X_{T}^{\pi})-Y_{T}^{\pi}|^2].
 	\end{aligned}
 \end{eqnarray}
In particular, if one of the following two conditions holds, we have $\epsilon=0$:
\begin{itemize}
\item[\rm (i)] Coefficient $b$, $\sigma$, $\beta$, $f$ and $g$ have bounded derivative function with $K$-Lipchitz derivatives.
\item[\rm (ii)] For each $y\in \mathbb R$ and $e\in E$, the map $x\in \mathbb R^d\to \beta(x,y,e)$ admits a Jacobian matrix $\nabla_{x}\beta(x,y,e)$ such that the function $a(x,y,\xi;e):=\xi^{\mathsf{T}}(\nabla_{x}\beta(x,y,e)+I_{d})\xi$ satisfies one of the following condition uniformly in $(x, y,\xi)\in{\mathbb{R}^d\times\mathbb R\times \mathbb{R}^d}$
\begin{eqnarray*}	
	a(x,y,\xi;e) \geq |\xi^2|K^{-1}~ {\hbox{or}}~a(x,y,\xi;e) \leq -|\xi^2|K^{-1},~ e\in E,~ (x, y)\in{\mathbb{R}^d\times\mathbb{R}^d}.
\end{eqnarray*}
\end{itemize}
 \end{theorem}

Theorem \ref{th5.1} allows us to state that the simulation error (left side of equation (\ref{eq3.1a})) of deep learning FBSDE method can be bounded through the value of the objective function (\ref{eq3.14 }) and the time discretization error $h^{1-\epsilon}$. It is illuminating to note that for $\beta=0$, that is, the no-jump case, it has been shown that $\epsilon=0$ for the deep BSDE method in \cite{HJ2020}. Note also that the appearance of $\epsilon$ is mainly due to the error $\sum_{n=0}^{N-1}\int_{t_{n}}^{t_{n+1}}\mathbb{E}|Z_{t}-Z_{t_{n}}^{\pi}|^2dt$. Under some additional conditions such as (i), (ii) in Theorem \ref{th5.1} or similar conditions (see, e.g., \cite{Buckdahn94}), we can prove that the estimate (\ref{eq3.1a}) holds true for $\epsilon=0$.

The constant $C$ in Theorem \ref{th5.1} depends on the constants in Assumptions 1, 2, 3, and 4, related to the coefficients $b$, $\sigma$, $\beta$, $f$, $g$, and $T$, but is independent of $d$, $h$, and $N$. As will be seen in the proof, roughly speaking, the weaker the coupling (resp., the stronger the monotonicity, the smaller the time horizon) is, the easier the condition is satisfied, and the smaller the constants $C$ related with error estimates are.

In what follows, we concentrate on the proof of Theorem \ref{th5.1}, which is divided in
the three sections. The next section will be of great utility in order to prove the convergence of Markovian iteration.

From now on we denote by $C$ a generic constant that only depends on $\mathbb E|\xi|^2$, $T$ and the coefficients $b$, $\sigma$, $\beta$, $f$, $g$, but is independent of $d$, $h$, $N$ and $x$. The value of $C$ may change from line to line when there is no need to distinguish.

\section{Convergence of Markovian iteration}
To prove Theorem \ref{th5.1}, we need several theorems and lemmas. Recalling the discrete form (\ref{eq3.16}) and taking conditional expectations  $\mathbb{ E}_{t_{n}}\{ \cdot \} $ on both sides of the last equation in (\ref{eq3.16}), we obtain
	\begin{eqnarray}\label{eq5.1}
	\hat{Y}^{\pi}_{t_{n}}=\mathbb{E}_{t_{n}}\{\hat{Y}^{\pi}_{t_{n+1}}+f(t_{n}, \hat{X}^{\pi}_{t_{n}}, \hat{Y}^{\pi}_{t_{n}}, \hat{Z}^{\pi}_{t_{n}}, \hat{\Gamma}^{\pi}_{t_{n}})\}.
	\end{eqnarray}
Multiplying $(\Delta  W_{t_{n}})^\mathsf{T}$ and taking condition expectations  $ \mathbb{E}_{t_{n}}\{ \cdot \} $ on both sides of  the last equation in (\ref{eq3.16}) again, one gets
\begin{eqnarray}\label{eq5.2}
\hat{Z}^{\pi}_{t_{n}}=\frac{1}{\Delta  t_{n}}\mathbb{E}_{t_{n}}\{\hat{Y}^{\pi}_{t_{n+1}}\Delta  W_{t_{n}}\}.
\end{eqnarray}
Similarly, multiplying $\int_{t_{n}}^{t_{n+1}}\int_{E}\gamma(e)\tilde{\mu}(de,dt)$ and taking condition expectations  $ \mathbb{E}_{t_{n}}\{ \cdot \} $ on both sides of  the last equation in (\ref{eq3.16}) again yield
\begin{eqnarray}\label{eq5.3}
\hat{\Gamma}^{\pi}_{t_{n}}=\frac{1}{\Delta  t_{n}}\mathbb{E}_{t_{n}}\left\{\hat{Y}^{\pi}_{t_{n+1}}\int_{E}\gamma(e)\tilde{\mu}(de, (t_{n}, t_{n+1}])\right\}.
\end{eqnarray}
From (\ref{eq3.16}) and (\ref{eq5.1})-(\ref{eq5.3}), we can get another discrete system without deep learning as follows:
 \begin{eqnarray}\label{eq5.4}
\left\{
\begin{array}{l}
\hat{X}^{\pi}_{0}=\xi,\vspace{2ex}\\
\begin{aligned}
\hat{X}^{\pi}_{t_{n+1}}=& \hat{X}^{\pi}_{t_{n}}+b(t_{n},\hat{X}^{\pi}_{t_{n}},\hat{Y}^{\pi}_{t_{n}})\Delta t_{n}+\sigma(t_{n},\hat{X}^{\pi}_{t_{n}},\hat{Y}^{\pi}_{t_{n}}) \Delta W_{t_{n}}  \\&+\int_{E}\beta(\hat{X}^{\pi}_{t_{n}},\hat{Y}^{\pi}_{t_{n}},e)\tilde{\mu}(de,(t_{n},t_{n+1}]),
\end{aligned}\\
\hat{Y}^{\pi}_{T}=g(\hat{X}^{\pi}_{T}),\qquad
\hat{Z}_{t_{n+1}}^{\pi}=\frac{1}{\Delta  t_{n}}\mathbb{E}_{t_{n}}\{\hat{Y}^{\pi}_{t_{n+1}}\Delta  W_{t_{n}}\},\vspace{2ex}\\
\hat{\Gamma}^{\pi}_{t_{n}}=\frac{1}{\Delta  t_{n}}\mathbb{E}_{t_{n}}\left\{\hat{Y}^{\pi}_{t_{n+1}}\int_{E}\gamma(e)\tilde{\mu}(de,(t_{n},t_{n+1}])\right\},\vspace{2ex}\\
\hat{Y}_{t_{n}}^{\pi}=\mathbb{E}_{t_{n}}\{\hat{Y}^{\pi}_{t_{n+1}}+f(t_{n}, \hat{X}^{\pi}_{t_{n}}, \hat{Y}^{\pi}_{t_{n+1}}, \hat{Z}^{\pi}_{t_{n}}, \hat{\Gamma}^{\pi}_{t_{n}})\}.
\end{array}
\right.
\end{eqnarray}
On the solution $(\hat{X}^{\pi}_{t_{n}}$, $\hat{Y}^{\pi}_{t_{n}}$, $\hat{Z}^{\pi}_{t_{n}}$, $\hat{\Gamma}^{\pi}_{t_{n}})$ of (\ref{eq5.4}), we have the following theorem whose proof will be given later.

According to \cite{Z1999} and \cite{BBP1997}, we know that under Assumption 1, the solution $(\hat{X}^{\pi}_{t_{n}}$, $\hat{Y}^{\pi}_{t_{n}}$, $\hat{Z}^{\pi}_{t_{n}}$, $\hat{\Gamma}^{\pi}_{t_{n}})$ to FBSDEJs (\ref{eq3.15 }) is unique. Although the algorithm (\ref{eq5.4}) is explicit with respect to $\hat{X}^{\pi}_{t_{n+1}}$ and $\hat{Y}_{t_{n}}^{\pi}$, it cannot be implemented directly because of their coupling. To decouple the equations (\ref{eq5.4}) in practical computation, we can introduce a ``Markovian" iteration \cite{BZ2008} with $ u_{t_n,0}^{\pi}=0$, $\forall 0\le n\ge N$,
\begin{eqnarray}\label{eq5.6}
\left\{
\begin{array}{l}
\hat{X}^{\pi}_{0,m}=\xi,\vspace{1ex}\\
\begin{aligned}
\hat{X}^{\pi}_{t_{n+1},m}= &\hat{X}^{\pi}_{t_{n},m}+b(t_{n},\hat{X}^{\pi}_{t_{n},m},u_{t_{n},m-1}^{\pi}(\hat{X}^{\pi}_{t_{n},m}))\Delta t_{n}\\&+\sigma(t_{n},\hat{X}^{\pi}_{t_{n},m},u_{t_{n},m-1}^{\pi}(\hat{X}^{\pi}_{t_{n},m})) \Delta W_{t_{n}} \\&+\int_{E}\beta(\hat{X}^{\pi}_{t_{n},m},u_{t_{n},m-1}^{\pi}(\hat{X}^{\pi}_{t_{n},m}),e)\tilde{\mu}(de,(t_{n},t_{n+1}]),\vspace{1ex}
\end{aligned}\\
\hat{Y}^{\pi}_{T,m}=g(\hat{X}^{\pi}_{T,m}),\qquad \hat{Z}_{t_{n+1},m}^{\pi}=\frac{1}{\Delta  t_{n}}\mathbb{E}_{t_{n}}\{\hat{Y}^{\pi}_{t_{n+1},m}\Delta  W_{t_{n}}\},\qquad m=1,2,\dotsc,\\
\hat{\Gamma}^{\pi}_{t_{n},m}=\frac{1}{\Delta  t_{n}}\mathbb{E}_{t_{n}}\left\{\hat{Y}^{\pi}_{t_{n+1},m}\int_{E}\gamma(e)\tilde{\mu}(de,(t_{n},t_{n+1}])\right\},\vspace{1ex}\\
\hat{Y}_{t_{n},m}^{\pi}=\mathbb{E}_{t_{n}}\{\hat{Y}^{\pi}_{t_{n+1},m}+f(t_{n}, \hat{X}^{\pi}_{t_{n},m}, \hat{Y}^{\pi}_{t_{n+1},m}, \hat{Z}^{\pi}_{t_{n},m}, \hat{\Gamma}^{\pi}_{t_{n},m})\},\vspace{1ex}\\
u_{t_{n},m}^{\pi}(\hat{X}^{\pi}_{t_{n},m})=\hat{Y}_{t_{n},m}^{\pi}.
\end{array}
\right.
\end{eqnarray}
For proving the convergence of the ``Markovian" iteration (\ref{eq5.6}), we estimate  $u_{t_{n},m+1}^{\pi}(\hat{X}^{\pi}_{t_{n},m}) - u_{t_{n},m}^{\pi}(\hat{X}^{\pi}_{t_{n},m})$ in terms of $\hat{X}^{\pi}_{t_{n},m+1}-\hat{X}^{\pi}_{t_{n},m}$, and then  $\hat{X}^{\pi}_{t_{n},m+1}-\hat{X}^{\pi}_{t_{n},m}$ in terms of $u_{t_{n},m}^{\pi}(\hat{X}^{\pi}_{t_{n},m})-u_{t_{n},m-1}^{\pi}(\hat{X}^{\pi}_{t_{n},m})$, and obtain
\begin{eqnarray*}
	\begin{aligned}
		\mathop{\sup}\limits_{x} |u_{t_{n},m+1}^{\pi}(x)-&u_{t_{n},m}^{\pi}(x)|^2\leq	\nu(L(u_{t_{n},m}^{\pi})) \mathop{\sup}\limits_{x}|u_ {t_{n},m}^{\pi}(x)-u_{t_{n},m-1}^{\pi}(x)|^2,
	\end{aligned}
\end{eqnarray*}
where $\nu(L(u_{t_{n},m}^{\pi}))$ depends on the coefficients of the equation and the Lipschitz constant $L(u^\pi_{t_{n},m})$ of $u^\pi_{t_{n},m}$. $L(u^\pi_{t_{n},m})$ will be estimated in following lemma, Lemma \ref{le5.3}.
If the ``Markovian" iteration converges, we need to control $\nu(L(u_{t_{n},m}^{\pi}))< 1$. In addition, we will show that $u^\pi_{t_{n},m}$ is linearly growing and satisfies
\begin{eqnarray}\label{eq5.7a}
	|u_{t_{n},m}^{\pi}(x)|^2\leq G(u_{t_{n},m}^{\pi})|x|^{2}+H(u_{t_{n},m}^{\pi}).
\end{eqnarray}
To estimate the Lipschitz constant $L(u^\pi_{t_{n},m})$ and the constants $G(u_{t_{n},m}^{\pi})$ and $H(u_{t_{n},m}^{\pi})$ in the linear growth condition (\ref{eq5.7a}), let us define
\begin{eqnarray*}
		\begin{aligned}
			L_{0}  =& [b_{y}+\sigma_{y}+\beta_{y}][g_{x}+f_{x}T]T \exp([2+2k_1+2k_2+\sigma_{x}+\beta_{x}+f_{z}+K^2_{\gamma}f_{\Gamma}]T\\&+[b_{y}+\sigma_{y}+\beta_{y}][g_{x}+f_{x}T]T),\\
			L_{1}  = &[g_{x}+f_{x}T] [\exp([2+2k_1+2k_2+\sigma_{x}+\beta_{x}+f_{z}+K^2_{\gamma}f_{\Gamma}]T\\&+[b_{y}+\sigma_{y}+\beta_{y}][g_{x}+f_{x}T]T+1) \vee 1],\\
c_{0}(L_{1}):=&T[g_{x}\Upsilon_{1}([1+2k_2+f_{z}+K^2_{\gamma}f_{\Gamma}]T,[1+2k_1+\sigma_{x}+\beta_{x}]T+[b_{y}+\sigma_{y}+\beta_{y}]L_{1}T)\\
&+f_{x}T\Upsilon_{0}([1+2k_2+f_{z}+K^2_{\gamma}f_{\Gamma}]T)\\
&\times\Upsilon_{0}([1+2k_1+\sigma_{x}+\beta_{x}]T+[b_{y}+\sigma_{y}+\beta_{y}]L_{1}T)],\vspace{2ex}\\
	c_{1}(L_{1}):=&[b_{y}+\sigma_{y}+\beta_{y}]c_{0}(L_{1}),\vspace{2ex}\\	L_{2}(L_{1}):=&[b_{0}+\sigma_{0}+\beta_{0}]c_{0}(L_{1})+e^{[1+2k_2+f_{z}+K^2_{\gamma}f_{\Gamma}]^{+}T}g_{0}\\
&+f_{0}T\Upsilon_{0}([1+2k_2+f_{z}+K^2_{\gamma}f_{\Gamma}]T),\\
	\Upsilon_{0}(x):=&\frac{e^{x}-1}{x},~~(x>0),\qquad	\Upsilon_{1}(x,y):=\mathop{\sup}\limits_{0<\theta<1}\theta e^{\theta x}\Upsilon_{0}(\theta y).
\end{aligned}
	\end{eqnarray*}

Then we have the following lemma.
\begin{lemma}\label{le5.3}
	If
	\begin{eqnarray} \label{eq5.7}
		L_{0}\leq e^{-1}, \quad c_{1}(L_{1})<c_{1}<1,
	\end{eqnarray}
then for any $\bar{L}>L_{1}$, $\bar{G}>L_{1}$, $L_{2}>L_{2}(L_{1})$, and for $h$ small enough, we have
		\begin{eqnarray*}
		L(u_{t_{n},m}^{\pi})\leq \bar{L}, \quad	G(u_{t_{n},m}^{\pi})	\leq \bar{G}, \quad H(u_{t_{n},m}^{\pi})\leq  \frac{L_{2}}{1-c_{1}}.
	\end{eqnarray*}	
\end{lemma}
  \begin{remark}
  If any of the five conditions of Assumption 3 holds true, then \eqref{eq5.7} hold.
  \end{remark}

\subsection{Estimates for the difference of solutions} In order to prove Lemma \ref{le5.3}, we consider the following system of equations
 \begin{align}\label{eq5.8}
 \hat{X}^{\pi}_{t_{n+1}}=& \hat{X}^{\pi}_{t_{n}}+b(t_{n}, \hat{X}^{\pi}_{t_{n}}, \varphi(\hat{X}^{\pi}_{t_n}))\Delta t_{n}+\sigma(t_{n},\hat{X}^{\pi}_{t_{n}},\varphi(\hat{X}^{\pi}_{t_n})) \Delta W_{t_{n}} \vspace{2ex} \nonumber\\ &+\int_{E}\beta(t_{n},\hat{X}^{\pi}_{t_{n}},\varphi(\hat{X}^{\pi}_{t_n}),e)\tilde{\mu}(de,(t_{n},t_{n+1}]),\\
 \label{eq5.9}	
  \hat{Y}_{t_{n+1}}^{\pi}=&\hat{Y}_{t_{n}}^{\pi}-f(t_{n},\hat{X}^{\pi}_{t_{n}},\hat{Y}_{t_{n}}^{\pi},\hat{Z}_{t_{n}}^{\pi},\hat{\Gamma}_{t_n}^{\pi})\Delta t_{n}+ \int_{t_{n}}^{t_{n+1}} Z_{t_{n}}^{i}d W_{t}\nonumber\\&+\int_{t_{n}}^{t_{n+1}}\int_{E}U_{t_{n}}^{i}(e)\tilde{\mu}(de,dt),
 \end{align}
with $\hat{X}^{\pi}_{t_{0}}=\xi$, and
 		\begin{eqnarray*}
 		\begin{aligned}	
 			\hat{Z}_{t_{n}}^{\pi}&:=\frac{1}{\Delta  t_{n}}\mathbb{E}_{t_{n}}\{\hat{Y}^{\pi}_{t_{n+1}}\Delta  W_{t_{n}}\},\quad
 			\hat{\Gamma}^{\pi}_{t_{n}}:=\frac{1}{\Delta  t_{n}}\mathbb{E}_{t_{n}}\left\{\hat{Y}^{\pi}_{t_{n+1}}\int_{E}\gamma(e)\tilde{\mu}(de,(t_{n},t_{n+1}])\right\},
 		\end{aligned}
 	\end{eqnarray*}
 where $ \varphi $ is uniformly Lipschitz continuous with $ L(\varphi) $ denoting the square of the Lipschitz constant of $ \varphi$.

Since the terminal condition of $\hat{Y}_{T}^{\pi}$ is not specified, the system of (\ref{eq5.8})-(\ref{eq5.9}) has infinitely many solutions. The difference between two such solutions is bounded by the following lemma.
 \begin{lemma}[Estimates for the difference of two solutions]\label{le5.4} 	
For $i=1,~2$, let $(\hat{X}^{\pi,i}_{t_{n}}$, $\hat{Y}^{\pi,i}_{t_{n}}$, $\hat{Z}^{\pi,i}_{t_{n}}$, $\hat{\Gamma}^{\pi,i}_{t_{n}})$	be two solutions of \eqref{eq5.8}-\eqref{eq5.9}, with $\varphi$ replacing by $ \varphi^{i} $ and $\hat{X}^{\pi,i}_{t_{n}}$, $\hat{Y}^{\pi,i}_{t_{n}}\in L^2(\Omega,\mathcal F_{t_n},P)$. Suppose condition \eqref{eq2.2} holds. For sufficiently small $h$, denote
 	\begin{align*} 			
 			A_{1}:=&2k_1+\sigma_{x}+\lambda_1+b_{x}h+\beta_{x}, \\
 			A_{2}:=&\lambda_{1}^{-1}b_{y}+\sigma_{y}+\beta_{y}+b_{y}h,\\ 			A_{3}:=&\lambda_{2}+\lambda_{4}+(1+\lambda^{-1}_{2}+\frac{1}{K_{\gamma}^{2}}\int_{E}\gamma^2(e)\lambda(de)\lambda^{-1}_{3})f_{z}\Delta  t_{n},\\ 			A_{4}:=&\lambda_{3}+K_{\gamma}^{2}\lambda_{5}+(1+\lambda^{-1}_{2}+\frac{1}{K_{\gamma}^{2}}\int_{E}\gamma^2(e)\lambda(de)\lambda^{-1}_{3})f_{\Gamma}\Delta  t_{n},\\ 			A_{5}:=&1+2k_2+\lambda_{4}^{-1}f_{z}+\lambda_{5}^{-1}f_{\Gamma}
 +(1+\lambda^{-1}_{2}+\frac{1}{K_{\gamma}^{2}}\int_{E}\gamma^2(e)\lambda(de)\lambda^{-1}_{3})f_{y}h,\\	
 A_{6}:=&f_{x}+\left(1+\lambda^{-1}_{2}+\frac{\lambda^{-1}_{3}}{K_{\gamma}^{2}}\int_{E}\gamma^2(e)\lambda(de)\right)f_{x}h. 	
 	 	\end{align*}
where $\lambda_i>0$, $i=1,\dotsc,5$, are chosen such that
\begin{align}\label{eq4.10}
A_3\le 1 \quad {\hbox{and}}\quad A_4\le 1.
\end{align}
Then for $0 \leq n \leq N$ and any $\lambda_{6}>0$, we have
  \begin{align} \label{eq5.10}
  \mathbb{E}_{t_{n}}{| \Delta  \hat{X}_{n+1}|^2}\leq&[1+A_{1}h+(1+\lambda_{6})A_{2}hL(\varphi^{1})]| \Delta \hat{X}_{n}|^2 \vspace{2ex}\nonumber\\&+(1+\lambda_{6}^{-1})A_{2}h |\varphi^{1}(\hat{X}^{\pi,2}_{n})-\varphi^{2}(\hat{X}^{\pi,2}_{n})|^2,\\
  \label{eq5.11}
  |\Delta  \hat{Y}_{n}|^2+(1-A_{3})\Delta t_{n}|\Delta \hat{Z}_{n}|^2+&\frac{1}{K_{\gamma}^{2}}(1-A_{4})\Delta t_{n}| \Delta \hat{\Gamma}_{n}|^2\nonumber\\ \leq& (1+A_{5}h)\mathbb{E}_{t_{n}}{|\Delta  \hat{Y}_{n+1}|^2}+A_{6}h| \Delta \hat{X}_{n}|^2,
  \end{align}
  where
  \begin{eqnarray*}
  	\begin{aligned}
  		\Delta \hat{X}_{n}:&= \hat{X}^{\pi,1}_{t_n} - \hat{X}^{\pi,2}_{t_n}, \quad
  		\Delta  \hat{Y}_{n}:= \hat{Y}^{\pi,1}_{t_n} -\hat{Y}^{\pi,2}_{t_n},\\
  		\Delta \hat{Z}_{n}:&=\hat{Z}^{\pi,1}_{t_n}-\hat{Z}^{\pi,2}_{t_n},\quad
  		\Delta \hat{\Gamma}_{n}:=\hat{\Gamma}^{\pi,1}_{t_n}-\hat{\Gamma}^{\pi,2}_{t_n}.
  	\end{aligned}
  \end{eqnarray*}
 \end{lemma}
\begin{proof} For simplicity, let us denote
 	\begin{eqnarray*}
 	\begin{aligned}	
 		\Delta b_{n}=&b(t_{n}, \hat{X}^{\pi,1}_{t_{n}},\varphi^{1}(\hat{X}^{\pi,1}_{t_n}))- b(t_{n}, \hat{X}^{\pi,2}_{t_{n}}, \varphi^{2}(\hat{X}^{\pi,2}_{t_n})),\\
 	 		\Delta \sigma_{n}=&\sigma(t_{n},\hat{X}^{\pi,1}_{t_{n}},\varphi^{1}(\hat{X}^{\pi,1}_{t_n})) -\sigma(t_{n},\hat{X}^{\pi,2}_{t_{n}},\varphi^{2}(\hat{X}^{\pi,2}_{t_n})),\\ 	
 		\Delta \beta_{n}=&\beta(t_{n}, \hat{X}^{\pi,1}_{t_{n}},\varphi^{1}(\hat{X}^{\pi,1}_{t_n}), e)-\beta(t_{n}, \hat{X}^{\pi,2}_{t_{n}},\varphi^{2}(\hat{X}^{\pi,2}_{t_n}), e),\\ 	
 		\Delta f_{n}=&f(t_{n},\hat{X}^{\pi,1}_{t_{n}},\hat{Y}_{t_{n}}^{\pi,1},Z_{t_{n}}^{\pi,1},\Gamma_{t_{n}}^{\pi,1})
 -f(t_{n},\hat{X}^{\pi,2}_{t_{n}},\hat{Y}_{t_{n}}^{\pi,2},Z_{t_{n}}^{\pi,2},\Gamma_{t_{n}}^{\pi,2}),\\ 	
 		\Delta Z_{n}:=&Z_{t_n}^{1}-Z_{t_n}^{2},\quad
 		\Delta U_{n}:=U_{t_{n}}^{1}-U_{t_{n}}^{2}.
 	\end{aligned}
 \end{eqnarray*}
Then from (\ref{eq5.8}), we obtain
 	\begin{eqnarray}\label{eq4.13}
 	\begin{aligned}
 		\mathbb{E}_{t_{n}}\{|\Delta  \hat{X}_{n+1}|^2\}=&\mathbb{E}_{t_{n}}\left\{\left| \Delta  \hat{X}_{n}+ \Delta b_{n}\Delta  t_{n} + \Delta \sigma_{n} \Delta  W_{t_{n}}+\int_{E}	\Delta \beta_{n}\tilde{\mu}(de,(t_{n},t_{n+1}])\right|^2 \right\}
 		\\ 
 		 \leq& [1+(b_{x}\Delta  t_{n}+\sigma_{x}+2k_1+\lambda_1+\beta_{x})\Delta  t_{n}]\mathbb{E}_{t_{n}}\left\{|\Delta   \hat{X}_{n}|^2\right\}\\
 &+(\lambda_1^{-1}b_{y}+\sigma_{y}+b_{y}\Delta  t_{n}+\beta_{y})\Delta  t_{n}\mathbb{E}_{t_{n}}\left\{|\varphi^{1}(\hat{X}^{\pi,1}_{t_n})-\varphi^{2}(\hat{X}^{\pi,2}_{t_n})|^2\right\} \\
 		\leq& (1+A_{1}h)\mathbb{E}_{t_{n}}\left\{|\Delta  \hat{X}_{n}|^2\right\}\\&
 +A_{2}h \mathbb{E}_{t_{n}}\left\{|\varphi^{1}(\hat{X}^{\pi,1}_{t_n})
 -\varphi^{1}(\hat{X}^{\pi,2}_{t_n})+\varphi^{1}(\hat{X}^{\pi,2}_{t_n})
 -\varphi^{2}(\hat{X}^{\pi,2}_{t_n})|^2\right\} \\
 		\leq& [1+A_{1}h+(1+\lambda_{6})A_{2}hL(\varphi^{1})]\mathbb{E}_{t_{n}}\left\{| \Delta  \hat{X}_{n}|^2\right\}\\&+(1+\lambda_{6}^{-1})A_{2}h \mathbb{E}_{t_{n}}\left\{|\varphi^{1}(\hat{X}^{\pi,2}_{t_n})-\varphi^{2}(\hat{X}^{\pi,2}_{t_n})|^2\right\},
 	\end{aligned}
 \end{eqnarray}
and therefore (\ref{eq5.10}) is proved.

Similarly, from (\ref{eq5.9}), we have
 	\begin{eqnarray}\label{eq5.12}
 \Delta \hat{Y}_{n}+\int_{t_{n}}^{t_{n+1}}\Delta  Z_{n} dW_{n}+\int_{t_{n}}^{t_{n+1}}\int_{E} \Delta U_{n}(e) \tilde{\mu}(de,dt)= \Delta \hat{Y}_{n+1} + \Delta  f_{n} \Delta t_{n}.	
 \end{eqnarray}
Squaring and taking conditional expectation on both sides of equation  (\ref{eq5.12}), we obtain
	\begin{eqnarray}\label{eq5.012}
\begin{aligned}
|\Delta  \hat{Y}_{n}|^2&+\mathbb{E}_{t_{n}}\left\{\int_{t_{n}}^{t_{n+1}}|\Delta  Z_{n}|^2dt\right\}+\mathbb{E}_{t_{n}}\left\{\int_{t_{n}}^{t_{n+1}}\int_{E}|\Delta  U_{n}(e)|^2\lambda(de)dt\right\}\\&=\mathbb{E}_{t_{n}}\{|\Delta  \hat{Y}_{n+1}|^2 + 2\Delta \hat{Y}_{n+1}\Delta  f_{n}\Delta  t_{n}+|\Delta  f_{n}|^2\Delta  t_{n}^2\}.	
\end{aligned}
\end{eqnarray}
Next we estimate the second and third terms on the left-hand side of (\ref{eq5.012}). The second term can be estimated by the following inequality, in view of Cauchy-Schwarz inequality,
 	\begin{eqnarray}\label{eq3.31}
 \begin{aligned}
 \mathbb{E}_{t_{n}}\left\{\int_{t_{n}}^{t_{n+1}}|\Delta  Z_{n}|^2dt\right\}& \geq \Delta  t_{n} |\Delta  \hat{Z}_{n}|^2+2\Delta  t_{n} \Delta  \hat{Z}_{n}\mathbb{E}_{t_{n}}\{\Delta f \Delta  W_{t_{n}}\}\\ &\geq (1-\lambda_{2})\Delta t_{n}|\Delta  \hat{Z}_{n}|^2-\lambda_{2}^{-1}\Delta t_{n}^{2}\mathbb{E}_{t_{n}}\{|\Delta  f_{n}|^2\}.
 \end{aligned}
 \end{eqnarray}
To estimate the third term on the left-hand side of (\ref{eq5.012}), we use the condition \eqref{eq2.2} to get
 \begin{align*}
 \int_{t_{n}}^{t_{n+1}}\int_{E}|\gamma(e)\Delta  U_{n}(e)|^2\lambda(de)dt \leq K_{\gamma}^2\int_{t_{n}}^{t_{n+1}}\int_{E}|\Delta  U_{n}(e)|^2\lambda(de)dt.
 \end{align*}
As a consequence, we have
  	\begin{align}\label{eq5.16a}	
	\mathbb{E}_{t_{n}}\left\{\int_{t_{n}}^{t_{n+1}}\int_{E}|\Delta  U_{n}(e)|^2\lambda(de)dt\right\} & \geq \frac{1}{K^2_{\gamma}}\mathbb{E}_{t_{n}}\left\{\int_{t_{n}}^{t_{n+1}}\int_{E}|\gamma(e)\Delta  U_{n}(e)|^2\lambda(de)dt\right\}
\nonumber\\& \geq \frac{1}{\Delta  t_{n}K_{\gamma}^2}\mathbb{E}_{t_{n}}\left\{\left| \int_{t_{n}}^{t_{n+1}}\int_{E}\gamma(e)\Delta  U_{n}(e)\lambda(de)dt\right|^2\right\}.
	\end{align}	
On the other hand, from (\ref{eq5.12}), we also have
	\begin{eqnarray}\label{eq5.17a}
	\begin{aligned}	
		\mathbb{E}_{t_{n}}&\left\{\int_{t_{n}}^{t_{n+1}}\int_{E}\gamma(e)\Delta  U_{n}(e)\lambda(de)dt\right\} \\&= \mathbb{E}_{t_{n}}\left\{\int_{t_{n}}^{t_{n+1}}\int_{E}\Delta  U_{n}(e)\tilde{\mu}(de,dt)\int_{E}\gamma(e)\tilde{\mu}(de,(t_{n},t_{n+1}])\right\}
		\\& = \mathbb{E}_{t_{n}}\{[\Delta  \hat{Y}_{n+1}+\Delta  f_{n} \Delta  t_{n}]\int_{E}\gamma(e)\tilde{\mu}(de,(t_{n},t_{n+1}])\}\\ &= \Delta  t_{n}\left[\Delta  \hat{\Gamma}_{n}+\mathbb{E}_{t_{n}}\left\{\Delta  f_{n}\int_{E}\gamma(e)\tilde{\mu}(de,(t_{n},t_{n+1}])\right\}\right].
	\end{aligned}	
\end{eqnarray}
Substituting (\ref{eq5.17a}) into (\ref{eq5.16a}) yields
	\begin{eqnarray}\label{eq3.32}
\begin{aligned}	
\mathbb{E}_{t_{n}}&\left\{\int_{t_{n}}^{t_{n+1}}\int_{E}|\Delta U_{n}(e)|^2 \lambda(de)dt\right\}
\\ &\geq \frac{\Delta  t_{n}}{K_{\gamma}^2} |\Delta  \hat{\Gamma}_{n}|^2+\frac{2\Delta  t_{n}}{K_{\gamma}^2}\Delta \hat{\Gamma}_{n}\mathbb{E}_{t_{n}}\left\{\Delta  f_{n}\int_{E}\gamma(e)\tilde{\mu}(de,(t_{n},t_{n+1}])\right\}
\\ &\geq \frac{1}{K_{\gamma}^2}(1-\lambda_{3})\Delta  t_{n} |\Delta \hat{\Gamma}|^2-\lambda_{3}^{-1}\frac{1}{K_{\gamma}^2}\int_{E}\gamma^2(e)\lambda(de) \Delta  t_{n}^2\mathbb{E}_{t_{n}}\{|\Delta  f_{n}|^2\}.
\end{aligned}	
\end{eqnarray}
Combining (\ref{eq5.012}), (\ref{eq3.31}) and (\ref{eq3.32}) leads to
	\begin{eqnarray}\label{eq5.19a}
	\begin{aligned}	
		|\Delta  \hat{Y}_{n}|^2 +&(1-\lambda_{2})\Delta  t_{n}|\Delta  \hat{Z}_{n}|^2 + \frac{1}{K_{\gamma}^2}(1-\lambda_{3})\Delta  t_{n}|\Delta \hat{\Gamma}_{n}|^2 \\ \leq & \mathbb{E}_{t_{n}}\left\{|\Delta \hat{Y}_{n+1}|^2+2\Delta  \hat{Y}_{n+1}\Delta  f_{n}\Delta  t_{n}\right.\\
&\left.+\left(1+\lambda_{2}^{-1}+\frac{\lambda_{3}^{-1}}{K_{\gamma}^2}\int_{E}\gamma^2(e)\lambda(de)\right)\Delta  t_{n}^2|\Delta  f_{n}|^2\right\}.
	\end{aligned}	
\end{eqnarray}
Finally, we substitute
	\begin{eqnarray*}
	\begin{aligned}	
		2\Delta   \hat{Y}_{n+1}  \Delta  f_n	\leq& f_{x}|\Delta  \hat{X}_{n}|^2+(1+2k_2+\lambda_{4}^{-1}f_{z}+\lambda_{5}^{-1}f_{\Gamma}) |\Delta  \hat{Y}_{n+1}|^2\\
&+\lambda_{4}|\Delta  \hat{Z}_{n}|^2+\lambda_{5}|\Delta  \hat{\Gamma}_{n}|^2
	\end{aligned}
\end{eqnarray*}
into (\ref{eq5.19a}) and obtain (\ref{eq5.11}). The proof is completed.	
\end{proof}

We need the following a priori estimates for the solution of \eqref{eq5.8}-\eqref{eq5.9}.
 \begin{lemma}[A priori estimates]\label{le5.5}
Let $(\hat{X}^{\pi}_{t_{n}}$, $\hat{Y}^{\pi}_{t_{n}}$, $\hat{Z}^{\pi}_{t_{n}}$, $\hat{\Gamma}^{\pi}_{t_{n}})$ be the solution of \eqref{eq5.8}-\eqref{eq5.9} with $\hat{Y}_{T}^{\pi}=g(\hat X_T^\pi)$ and $\varphi$ replacing by $ \varphi_n $, where $ \varphi_{n} $ is linearly growing functions satisfying
	\begin{eqnarray*}
	|\varphi_{n}(x)|^2 \leq G(\varphi_{n})|x|^2+H(\varphi_{n}).
\end{eqnarray*}
Let
	\begin{eqnarray*}
	G(\varphi):=\sup_{n\ge 0} G(\varphi_{n}), \quad  H(\varphi):=\sup_{n\ge 0} H(\varphi_{n}).
\end{eqnarray*}
Suppose condition \eqref{eq2.2} holds. Then for $0 \leq n \leq N$, we have
\begin{align}\label{eq5.13}
\mathbb{E}_{t_{n}}{|\hat{X}_{t_{n+1}}^{\pi}|^2}\leq&[1+A_{1}h+A_{2}h G(\varphi)]|\hat{X}_{t_n}^{\pi}|^2\nonumber\\
&+[b_{0}+\sigma_{0}+\beta_{0}+b_{0}h+A_{2}H({\varphi})] h,\\
\label{eq5.14}
|\hat{Y}_{t_n}^{\pi}|^2+&(1-A_{3})\Delta t_{n}|\hat{Z}_{t_n}^{\pi}|^2+\frac{1}{K_{\gamma}^{2}}(1-A_{4})\Delta t_{n}|\hat{\Gamma}_{t_n}|^2 \nonumber\\ \leq& (1+A_{5}h)\mathbb{E}_{t_{n}}{|\hat{Y}_{t_{n+1}}^{\pi}|^2}+A_{6}h |\hat{X}_{t_{n}}|^2\nonumber\\&+\left[1+(1+\lambda_{2}^{-1}+\lambda_{3}^{-1}\frac{1}{K_{\gamma}^{2}}\int_{E}\gamma^2(e)\lambda(de))h  \right]f_{0}h.
\end{align}
\end{lemma}
\begin{proof} From (\ref{eq5.8}) and the assumption of the lemma, one gets
		\begin{align*}	
\mathbb{E}_{t_{n}}{|\hat{X}^{\pi}_{t_{n+1}}|^2}
=& \mathbb{E}_{t_{n}}\left|\hat{X}^{\pi}_{t_{n}}+b(t_{n}, \hat{X}^{\pi}_{t_{n}}, \varphi_n(\hat{X}^{\pi}_{t_n}))\Delta  t_{n}+\sigma(t_{n},\hat{X}^{\pi}_{t_{n}},\varphi_n(\hat{X}^{\pi}_{t_n})) \Delta  W_{t_{n}}\right.\\& \left.+\int_{E}\beta(\hat{X}^{\pi}_{t_{n}},\varphi_n(\hat{X}^{\pi}_{t_n}),e)\tilde{\mu}(de,(t_{n},t_{n+1}])\right|^2
\\\leq&
\mathbb{E}_{t_{n}}|\hat{X}^{\pi}_{t_{n}}|^2+2\Delta  t_{n}\mathbb{E}_{t_{n}}\left\{\left(b(t_{n}, \hat{X}^{\pi}_{t_{n}}, \varphi_n(\hat{X}^{\pi}_{t_n}))\right)^T\hat{X}^{\pi}_{t_{n}}\right\}\\&+\mathbb{E}_{t_{n}}\left\{b^{2}(t_{n}, \hat{X}^{\pi}_{t_{n}}, \varphi_n(\hat{X}^{\pi}_{t_n}))(\Delta  t_{n})^{2}\right\}+\mathbb{E}_{t_{n}}\left\{\sigma^2(t_{n}, \hat{X}^{\pi}_{t_{n}}, \varphi_n(\hat{X}^{\pi}_{t_n}))\Delta  t_{n}\right\}
\\&+\int_{E}\beta^{2}(\hat{X}^{\pi}_{t_{n}},\varphi_n(\hat{X}^{\pi}_{t_n}),e)\lambda(de)\Delta t_{n}
\\ \leq& [1+A_{1}h] |\hat{X}_{t_n}^{\pi}|^2+A_{2}h|\varphi_n(\hat{X}^{\pi}_{t_n})|^2+[b_{0}+\sigma_{0}+\beta_{0}+b_{0}h]h
\\
\leq& [1+A_{1}h+A_{2}h G(\varphi)]|\hat{X}_{t_n}^{\pi}|^2+[b_{0}+\sigma_{0}+\beta_{0}+b_{0} h+A_{2}H(\varphi)] h,
\end{align*}	
which implies that (\ref{eq5.13}) holds.

Next, applying Lemma \ref{le5.4} yields
	\begin{eqnarray*}
	\begin{aligned}	
		|\hat{Y}_{t_n}|^2 &+(1-\lambda_{2})\Delta  t_{n}| \hat{Z}_{t_n}|^2 + \frac{1}{K_{\gamma}^2}(1-\lambda_{3})\Delta  t_{n}| \hat{\Gamma}_{t_n}|^2 \\ & \leq \mathbb{E}_{t_{n}}\left\{|\hat{Y}_{t_{n+1}}|^2+2  \hat{Y}_{t_{n+1}} f\Delta  t_{n}+(1+\lambda_{2}^{-1}\right.\\
&\left.+\frac{1}{K_{\gamma}^2}\lambda_{3}^{-1}\int_{E}\gamma^2(e)\lambda(de))\Delta  t_{n}|f_{t_{n}}|^2\right\}.	
\end{aligned}	
\end{eqnarray*}
In view of
	\begin{eqnarray*}
	\begin{aligned}		
		\hat{Y}_{n+1} f_n	 &\leq f_{x}| \hat{X}_{n}|^2+(1+2k_2+\lambda_{4}^{-1}f_{z}+\lambda_{5}^{-1}f_{\Gamma}) | \hat{Y}_{n+1}|^2+\lambda_{4}| \hat{Z}_{n}|^2+\lambda_{5}| \hat{\Gamma}_{n}|^2.
	\end{aligned}	
\end{eqnarray*}
using (iii) of Assumption 1, we obtain (\ref{eq5.14}) and therefore complete the proof.
\end{proof}

Now we employ Lemmas \ref{le5.4} and \ref{le5.5} to show Lemma \ref{le5.3}.
\begin{proof}[of Lemma \ref{le5.3}]. To show Lemma \ref{le5.3}, let us introduce	
\begin{align*}
&\hat{X}^{\pi,1}_{t_n}=\hat{X}^{\pi}_{t_{n},m+1},&\quad &\hat{Y}^{\pi,1}_{t_n}=\hat{Y}^{\pi}_{t_{n},m+1},&\quad &\varphi^{1}(\hat{X}^{\pi,1}_{t_n})=u^{\pi}_{t_{n},m-1}(\hat{X}^{\pi,2}_{t_n}),&
\\
&\hat{X}^{\pi,2}_{t_n}= \hat{X}^{\pi}_{t_{n},m},&\quad
&\hat{Y}^{\pi,2}_{t_n}= \hat{Y}^{\pi}_{t_{n},m},&\quad &\varphi^{2}(\hat{X}^{\pi,2}_{t_n})=u^{\pi}_{t_{n},m}(\hat{X}^{\pi,1}_{t_n}).&
\end{align*}
In view of the above notations, we can set $\lambda_{6}=0$ in (\ref{eq4.13}) and therefore obtain from Lemmas \ref{le5.4} and \ref{le5.5} that
	\begin{align*}
\mathbb{E}_{t_{n}}{| \Delta  X_{n+1}|^2}\leq[1+A_{1}h+A_{2} hL(u_{t_{n},m-1}^{\pi})]| \Delta  X_{n}|^2,
\end{align*}
and
\begin{align*}
  |\Delta  Y_{n}|^2+(1-A_{3})\Delta  t_{n}| \Delta  Z_{n}|^2+&\frac{1}{K_{\gamma}^{2}}(1-A_{4})\Delta  t_{n}| \Delta  \Gamma_{n}|^2\\&\leq (1+A_{5}h)\mathbb{E}_{t_{n}}{|\Delta  Y_{n+1}|^2}+A_{6}h|\Delta  X_{n}|^2.
\end{align*}
In view of (\ref{eq4.10}) and $|\Delta  Y_{n}|=|u^{\pi}_{t_{n},m+1}-u^{\pi}_{t_{n},m}|$, we can get
	\begin{align*}
\begin{aligned}
|\Delta  u_{t_{n},m}^{\pi}|^2   \leq &(1+A_{5}h)\mathbb{E}_{t_{n}}{|\Delta  Y_{n+1}|^2}+A_{6}h | \Delta  X_{n}|^2 \\ \leq &[1+A_{5} h][1+A_{1} h+A_{2} hL(u_{t_{n},m-1}^{\pi})]L(u_{t_{n+1},m}^{\pi})|\Delta  X_{n}|^2\\&+A_{6}h|\Delta  X_{n}|^2.
\end{aligned}
\end{align*}
Then we have
	\begin{align*}
L(u_{t_{n},m}^{\pi})  \leq&  [1+A_{5}h][1+A_{1} h+A_{2} hL(u_{t_{n},m-1}^{\pi})]L(u_{t_{n+1},m}^{\pi})+A_{6} h
\\ \leq& \left\{1+\left[A_{1}+A_{5}+A_{1}A_{5} h+(A_{2}+A_{2}A_{5}h)L(u_{t_{n},m-1}^{\pi})\right]h\right\} L(u_{t_{n+1},m}^{\pi})\\
&+A_{6}h	\\
\leq & \left\{1+\left[(A_{1}+A_{5}+A_{1}A_{5} h+(A_{2}+A_{2}A_{5}h)L(u_{t_{n},m-1}^{\pi}))\vee 0\right]h\right\}L(u_{t_{n+1},m}^{\pi})\\
&+A_{6}h.
\end{align*}
Applying discrete Gronwall inequality and $L(u_{t_{N},m}^{\pi})=g_{x}$ leads to
	\begin{eqnarray}\label{eq5.15}
\begin{aligned}
L(u_{t_{n},m}^{\pi})  \leq& [g_{x}+A_{6}T]\\& \times [\exp((A_{1}+A_{5}+A_{1}A_{5} h)T+(A_{2}+A_{2}A_{5}h)TL(u_{t_{n,m-1}}^{\pi})) \vee 1].
\end{aligned}
\end{eqnarray}
Multiplying $(A_{2}+A_{2}A_{5}\Delta _{t_{n}})T$ on the both sides of equation (\ref{eq5.15}) and using $L(u_{t_{n},0}^{\pi})=0$, we get
	\begin{eqnarray}\label{eq5.15a}
     \begin{aligned}
       (A_{2}&+A_{2}A_{5}h)TL(u_{t_{n},m}^{\pi})\\ &
       \begin{aligned}
          \leq &[A_{2}+A_{2}A_{5}h]T[g_{x}+A_{6}T] \\&\times\left\{\exp([A_{1}+A_{5}+A_{1}A_{5} h]T+[A_{2}+A_{2}A_{5}h]TL(u_{t_{n},m-1}^{\pi})) \vee 1\right\}
            \end{aligned}	
       \\ &
       \begin{aligned}
       \leq &[A_{2}+A_{2}A_{5}h]T[g_{x}+A_{6}T]\\& \times\left\{ \exp([A_{1}+A_{5}+A_{1}A_{5} h]T+[A_{2}+A_{2}A_{5}h][g_{x}+A_{6}T]T\right.\\
       &+[A_{2}+A_{2}A_{5}h][g_{x}+A_{6}T]T[\exp([A_{1}+A_{5}+A_{1}A_{5} h]T\\
       &\left.+[A_{2}+A_{2}A_{5}h]TL(u_{t_{n},m-2}^{\pi}))] ) \vee 1\right\}.
     \end{aligned}
      \end{aligned}
     \end{eqnarray}
Now we set 	\begin{eqnarray*}
	\begin{aligned}
		L_{0}(\lambda,h):= &[A_{2}+A_{2}A_{5}h]T[g_{x}+A_{6}T]\\ &\times \exp([A_{1}+A_{5}+A_{1}A_{5}h]T+[A_{2}+A_{2}A_{5}h][g_{x}+A_{6}T]T),\\	
		L_{1}(\lambda, h):=& [g_{x}+A_{6}T] [\exp([A_{1}+A_{5}+A_{1}A_{5}h]T\\&+[A_{2}+A_{2}A_{5}h][g_{x}+A_{6}T]T+1) \vee 1],
	\end{aligned}
\end{eqnarray*}
and choose
	\begin{eqnarray}\label{eq5.16}
\begin{aligned}
&\lambda_{2} :=\sqrt{\Delta  t_{n}},&  &\lambda_{4} :=1-\left(1+f_{z}+\frac{f_{z}}{K_{\gamma}^2}\int_{E}\gamma(e)^2\lambda(de)\right)\sqrt{\Delta  t_{n}}-f_{z}\Delta  t_{n},&
\\ &\lambda_{3} :=\sqrt{\Delta  t_{n}},& &\lambda_{5} :=\frac{1}{K_{\gamma}^2}\left[1-\left(1+f_{\Gamma}
+\frac{f_{\Gamma}}{K_{\gamma}^{2}}\int_{E}\gamma(e)^2\lambda(de)\right)\sqrt{\Delta  t_{n}}-f_{\Gamma }\Delta  t_{n}\right].&
\end{aligned}
\end{eqnarray}
Then $A_{3}=1$, $A_{4}=1$ and
	\begin{eqnarray*}
	\begin{aligned}
		\mathop{\lim}\limits_{h \to 0}&L_{0}(\lambda,h) =L_{0},\qquad
		\mathop{\lim}\limits_{h \to 0}L_{1}(\lambda,h)=L_{1}.
	\end{aligned}
\end{eqnarray*}
Hence, for any small enough $h$ and any
$\bar{L}>L_{1}$, if $ L_{0} \leq e^{-1} $, from (\ref{eq5.15a}) we have
	\begin{eqnarray*}
	\begin{aligned}
		L(u_{t_{n},m}^{\pi}) \leq& L_{1}(\lambda, h),
	\end{aligned}
\end{eqnarray*}
and therefore $L(u_{t_{n},m}^{\pi})\leq \bar{L}$.

Now, we study the linear growth condition (\ref{eq5.7a}) on $u_{t_{n},m}^{\pi}$. Similarly, it is easy get the following inequalities,
\begin{align*}
		G(u_{t_{n},m}^{\pi}) \leq& [g_{x}+A_{6}T][\exp([A_{1}+A_{5}+A_{1}A_{5} h]T+[A_{2}+A_{2}A_{5}h]TG(u_{t_{n},m-1}^{\pi})) \vee 1],\\
H(u_{t_{n},m}^{\pi}) \leq& [e^{A_{5}T} \vee 1]g_{0}+\left[f_{0}+f_{0} (1+\lambda_{2}^{-1}+\lambda_{3}^{-1}\frac{1}{K_{\gamma}^{2}}\int_{E}\gamma(e)^2\lambda(de))h\right]\Upsilon_{0}^{n}(A_{5})\\&+[b_{0}+\sigma_{0}+\beta_{0}+b_{0} h+A_{2}H(u_{t_{n},m-1}^{\pi})]\\& \times \left[g_{x}\Upsilon_{1}^{n}(A_{5},A_{1}+A_{2}G(u_{t_{n},m-1}^{\pi}))+A_{6}\Upsilon_{0}^{n}(A_{5})\Upsilon_{0}^{n}(A_{1}+A_{2}G(u_{t_{n},m-1}^{\pi}))\right]
\\=&c_{1}\left(\lambda,h,G(u_{t_{n},m-1}^{\pi})\right)H(u_{t_{n},m-1}^{\pi})+L_{2}\left(\lambda,h,G(u_{t_{n},m-1}^{\pi})\right).
\end{align*}
Here
\	\begin{eqnarray*}
\begin{aligned}
	c_{0}(\lambda,h,&G(u_{t_{n},m-1}^{\pi}))\\:=&	\left[g_{x}\Upsilon_{1}^{n}(A_{5},A_{1}+A_{2}G(u_{t_{n},m-1}^{\pi}))+A_{6}\Upsilon_{0}^{n}(A_{5})\Upsilon_{0}^{n}(A_{1}+A_{2}G(u_{t_{n},m-1}^{\pi}))\right],\\
	c_{1}(\lambda,h,&G(u_{t_{n},m-1}^{\pi})):=A_{2}	c_{0}(\lambda,h,G(u_{t_{n},m-1}^{\pi})),\\
	L_{2}(\lambda,h,&G(u_{t_{n},m-1}^{\pi}):=	[b_{0}+\sigma_{0}+\beta_{0}+b_{0} h]c_{0}(\lambda,h,G(u_{t_{n},m-1}^{\pi}))+[e^{A_{5}T} \vee 1]g_{0}\\&+\left[f_{0}+f_{0} (1+\lambda_{2}^{-1}+\lambda_{3}^{-1}\frac{1}{K_{\gamma}^{2}}\int_{E}\gamma(e)^2\lambda(de))h\right]\Upsilon_{0}^{n}(A_{5}).
\end{aligned}
\end{eqnarray*}
We still choose $\lambda_{i}$  as in (\ref{eq5.16}), $i=2, 3, 4, 5$. Noting that
	\begin{eqnarray*}
	\begin{aligned}
		\mathop{\lim}\limits_{h\to 0}&c_{1}(\lambda,h,G(u_{t_{n},m-1}^{\pi}))=c_{1}(G),\qquad
		\mathop{\lim}\limits_{h \to 0}L_{2}(\lambda,h,G(u_{t_{n},m-1}^{\pi}))=L_{2}(G),
	\end{aligned}
\end{eqnarray*}
because of $ L_{0} \leq e^{-1}$, $c_{1}(L_{1})\leq 1 $, for any small enough $ h$, and for any
$\bar{G}>L_{1}$, $c_{1}(\bar{G})\leq c_{1}<1$, $L_{2}(\bar{G})\leq L_{2}$, we have $G(u_{t_{n},m}^{\pi})\leq \bar{G}$, $H(u_{t_{n},m}^{\pi})\leq \frac{L_{2}}{1-c_{1}}$. This completes the proof of Lemma \ref{le5.3}.
\end{proof}

Now we define
\begin{align*}	c_{2}&(\lambda_{1},L_{1},L_{1})\\:=&\left[e^{[1+2k_2+\sigma_{x}+\beta_{x}+[b_{y}+\sigma_{y}+\beta_{y}]L_{1}]T}\vee 1\right](1+\lambda_{1}^{-1})[b_{y}+\sigma_{y}+\beta_{y}]T[g_{x}\\&
	\times \Upsilon_{1}\left([1+2k_2+f_{z}+f_{\Gamma}]T,[2+2k_1+\sigma_{x}+\beta_{x}]T+(1+\lambda_{1})[b_{y}+\sigma_{y}+\beta_{y}]L_{1}T\right)\\&
+f_{x}T\Upsilon_{0}([2+k_2+f_{z}+f_{\Gamma}]T)\Upsilon_{0}([1+k_1+\sigma_{x}+\beta_{x}]T\\
&+(1+\lambda_{1})[b_{y}+\sigma_{y}+\beta_{y}]L_{1}T)]\vspace{2ex},\end{align*}
and
\begin{align*}	
	c_{2}&(L_{1},L_{1}):=\mathop{\inf}\limits_{\lambda_{1}>1}c_{2}(\lambda_{1},L_{1},L_{1}).
	\end{align*}
Then we have the following theorem which implies the convergence of the Markovian iteration (\ref{eq5.6}).
\begin{theorem}\label{co5.6}
	Assume $L_{0}\leq e^{-1}$ holds true and
\begin{align}c_{2}(L_{1},L_{1})< 1.
\end{align}
\begin{itemize}
	\item[\rm (i)] For any $\bar{L}>L_{1}$, $\bar{G}>L_{1}$, $c_{1}(L_{1})<c_{1}<1$, $L_{2}>L_{2}(L_1)$, we have
	\begin{eqnarray*}
		L(u_{t_n}^{\pi})\leq \bar{L},\quad G(u_{t_n}^{\pi})	\leq \bar{G}, \quad H(u_{t_n}^{\pi})\leq \bar{H} =\frac{L_{2}}{1-c_{1}};
	\end{eqnarray*}
	
	\item[\rm (ii)] For any $c_{2}(L_{1},L_{1})<c_{2}\leq 1$, and sufficiently small $h$, when $m\to \infty$, we have
	\begin{eqnarray*}
		\mathop{\max}\limits_{0\leq n \leq N}|u^{\pi}_{t_{n},m}-u^{\pi}_{t_{n}}|^2\to 0.
	\end{eqnarray*}
\end{itemize}
\end{theorem}
\begin{proof} Employing Lemma 4.1, the proof of above theorem is similar to that of Theorem 5.1 in \cite{BZ2008}, we are not going to repeat this proof.
\end{proof}

\section{Error estimates for time discretization} We now study the error due to the time discretization. We first present the following theorem which gives the connections between coupled FBSDEJs and nonlinear PIDEs under weaker conditions.

\begin{theorem}\label{th5.7}
	Under Assumptions 1, 2 and 3, there exist a function $u:\mathbb{R} \times \mathbb{R}^{d} \to\mathbb{R}$ that satisfies the following statements
	\begin{enumerate}
	\item [\rm (i)] $|u(t,x_{1})-u(t,x_{2})|^2 \leq L_{1}|x_{1}-x_{2}|^2$, ~$|u(t,x)|^2 \leq L_{1}|x|^2+\frac{L_{2}(L_{1})}{1-c_{1}L_{1}}$.
\item [\rm (ii)] $|u(t,x)-u(s,x)|^2\le C(1+|x|^2)|t-s|$ for some constant $C$.
\item [\rm (iii)] $u$ is a viscosity solution of the PIDEs \eqref{eq2.1}.
    \item [\rm (iv)] The FBSDEJs \eqref{eq3.15 } has a unique solution $(X_{t},Y_{t},Z_{t},\Gamma_{t})$ and $Y_{t}=u(t,X_{t})$. Thus, $(X_{t},Y_{t},Z_{t},\Gamma_{t})$ also solve the following decoupled FBSDEJs
    \begin{eqnarray}\label{eq5.1a}
    	\left\{
    	\begin{aligned}    		X_t=&\xi+\int_{0}^{t}b(s,X_{s},u(s,X_{s}))ds+\int_{0}^{t}\sigma(s,X_{s},u(s,X_{s}))dW_{s}\\&+\int_{0}^{t}\int_{E}\beta(X_{s^-},u(s,X_{s}),e)\tilde{\mu}(de,ds),
    		\\Y_{t}=&g(X_{T})+\int_{t}^{T}f(s,X_t,Y_{s},Z_{s},\Gamma_{s})ds-\int_{t}^{T}Z_{s}dW_{s}\\&-\int_{t}^{T}\int_{E}U_{s}(e)\tilde{\mu}(de,ds).
    	\end{aligned}
    	\right.
    \end{eqnarray}
\item [\rm (v)] Furthermore, we have the following estimates:
\begin{align}\label{eq5.2a}
|\tilde{u}_{t_{n}}^{\pi}(t,x)-u(t_{n},x)|^2 \leq C[1+|\xi|^2]h,
\end{align}
and for any $\epsilon>0$
\begin{eqnarray}\label{eq5.26a}
	\begin{aligned}	
		\mathop{\max}\limits_{0\leq n<N}\mathop{\sup}\limits_{t_{n}\leq t \leq t_{n+1}}(&\mathbb{E}|X_{t}-\tilde{X}_{t_{n}}|^2+\mathbb{E}|Y_{t}-\tilde{Y}_{t_{n}}|^2)+\sum_{n=0}^{N-1}\int_{t_{n}}^{t_{n+1}}\mathbb{E}|Z_{t}-\tilde{Z}_{t_{n}}^{\pi}|^2dt \\ &+\sum_{n=0}^{N-1}\int_{t_{n}}^{t_{n+1}}\mathbb{E}|\Gamma_{t}-\tilde{\Gamma}_{t_{n}}^{\pi}|^2dt\leq C[1+\mathbb{E}|\xi|^2]h^{1-\epsilon},
	\end{aligned}
\end{eqnarray}
where $\tilde{u}_{t_{n}}^{\pi}(t,\tilde{X}_{t_{n}}^{\pi})=\tilde{Y}_{t_{n}}^{\pi}$ and
\begin{eqnarray}\label{eq5.27a}
	\left\{
	\begin{array}{l}
		\begin{aligned}
			\tilde{X}^{\pi}_{t_{n+1}}=& \tilde{X}^{\pi}_{t_{n}}+b(t_{n},\tilde{X}^{\pi}_{t_{n}},u(t_{n},\tilde{X}^{\pi}_{t_{n}}))\Delta t_{n}+\sigma(t_{n},\tilde{X}^{\pi}_{t_{n}},u(t_{n},\tilde{X}^{\pi}_{t_{n}})) \Delta W_{t_{n}}  \\&+\int_{E}\beta(\tilde{X}^{\pi}_{t_{n}},u(t_{n},\tilde{X}^{\pi}_{t_{n}}),e)\tilde{\mu}(de,(t_{n},t_{n+1}]),
		\end{aligned}\\
		\tilde{Z}_{t_{n+1}}^{\pi}=\frac{1}{\Delta  t_{n}}\mathbb{E}_{t_{n}}\{\tilde{Y}^{\pi}_{t_{n+1}}\Delta  W_{t_{n}}\},\vspace{2ex}\\
		\tilde{\Gamma}^{\pi}_{t_{n}}=\frac{1}{\Delta  t_{n}}\mathbb{E}_{t_{n}}\left\{\tilde{Y}^{\pi}_{t_{n+1}}\int_{E}\gamma(e)\tilde{\mu}(de,(t_{n},t_{n+1}])\right\},\vspace{2ex}\\
		\tilde{Y}_{t_{n}}^{\pi}=\mathbb{E}_{t_{n}}\{\tilde{Y}^{\pi}_{t_{n+1}}+f(t_{n}, \tilde{X}^{\pi}_{t_{n}}, \tilde{Y}^{\pi}_{t_{n}}, \tilde{Z}^{\pi}_{t_{n}}, \tilde{\Gamma}^{\pi}_{t_{n}})\}.
	\end{array}
	\right.
\end{eqnarray}
In particular, we have $\epsilon = 0 $ under the conditions (i) or (ii) in Theorem \ref{th5.1}.
     \end{enumerate}
\end{theorem}
\begin{proof}
 The proof of (i), (ii) and (iii) is similar to that of Theorem \ref{co5.6} and also similar to that of Theorem 6.1 of \cite{BZ2008}. Since (\ref{eq5.1a}) is decoupled, (iv) and (\ref{eq5.26a}) in (v) can be obtained directly from references \cite {E2007}. The estimate (\ref{eq5.2a}) in (v) can be obtained from (i), (ii) and (iv), and the proof is similar to that of Corollary 6.2 in Bender and Zhang  \cite{BZ2008}.
\end{proof}

Now, we are ready to derive the error estimates for standard time discretization.

\begin{theorem}[Error estimates for time discretization]\label{th5.2}
	Under assumptions 1, 2, 3 and 4, for sufficiently small $h$, equation \eqref{eq5.4} has a solution $(\hat{X}^{\pi}_{t_{n}}$, $\hat{Y}^{\pi}_{t_{n}}$, $\hat{Z}^{\pi}_{t_{n}}$, $\hat{\Gamma}^{\pi}_{t_{n}})$ such that
	\begin{eqnarray}\label{eq5.5}
	  \begin{aligned}
    \mathop{\max}\limits_{n<N}\mathop{\sup}\limits_{t_{n}\leq t \leq t_{n+1}}(\mathbb{E}|X_{t}-\hat{X}_{t_{n}}^{\pi}|^2&+\mathbb{E}|Y_{t}-\hat{Y}_{t_{n}}^{\pi}|^2)+\sum_{n=0}^{N-1}\int_{t_{n}}^{t_{n+1}}\mathbb{E}|Z_{t}-\hat{Z}_{t_{n}}^{\pi}|^2dt \\ &+\sum_{n=0}^{N-1}\int_{t_{n}}^{t_{n+1}}\mathbb{E}|\Gamma_{t}-\hat{\Gamma}_{t_{n}}^{\pi}|^2dt \leq C[1+\mathbb{E}|\xi|^2]h^{1-\epsilon}.
	  \end{aligned}
	\end{eqnarray}
In particular, $\epsilon = 0 $ under the conditions (i) or (ii) in Theorem \ref{th5.1}.
\end{theorem}
\begin{proof} Reviewing discrete schemes (\ref{eq5.4}) and (\ref{eq5.27a}), applying Cauchy-Schwarz inequality and (\ref{eq5.26a}), we have
\begin{align}\label{eq5.17}	
	\mathop{\max}\limits_{0\leq n<N}&\mathop{\sup}\limits_{t_{n}\leq t \leq t_{n+1}}(\mathbb E|X_{t}-\hat{X}_{t_{n}}|^2+\mathbb E|Y_{t}-\hat{Y}_{t_{n}}|^2) \nonumber \\
&+\sum_{n=0}^{N-1}\int_{t_{n}}^{t_{n+1}}(\mathbb E|Z_{t}-\hat{Z}_{t_{n}}^{\pi}|^2dt +\mathbb E|\Gamma_{t}-\hat{\Gamma}_{t_{n}}^{\pi}|^2)dt \nonumber \\ \leq&  \mathop{\max}\limits_{0\leq n<N}\mathop{\sup}\limits_{t_{n}\leq t \leq t_{n+1}}(\mathbb E|X_{t}-\tilde{X}_{t_{n}}|^2+\mathbb E|Y_{t}-\tilde{Y}_{t_{n}}|^2)+\sum_{n=0}^{N-1}\int_{t_{n}}^{t_{n+1}}\mathbb E|Z_{t}-\tilde{Z}_{t_{n}}^{\pi}|^2dt \nonumber\\ &+\sum_{n=0}^{N-1}\int_{t_{n}}^{t_{n+1}}\mathbb E|\Gamma_{t}-\tilde{\Gamma}_{t_{n}}^{\pi}|^2dt+ \mathop{\max}\limits_{0\leq n<N}\mathop{\sup}\limits_{t_{n}\leq t \leq t_{n+1}}(\mathbb{E}|\tilde{X}^{\pi}_{t_{n}}-\hat{X}^{\pi}_{t_{n}}|^2+\mathbb{E}|\tilde{Y}^{\pi}_{t_{n}}-\hat{Y}^{\pi}_{t_{n}}|^2)\\&+\sum_{n=0}^{N-1}\int_{t_{n}}^{t_{n+1}}\mathbb{E}|\tilde{Z}^{\pi}_{t}-\hat{Z}_{t_{n}}^{\pi}|^2dt +\sum_{n=0}^{N-1}\int_{t_{n}}^{t_{n+1}}\mathbb{E}|\tilde{\Gamma}^{\pi}_{t}-\hat{\Gamma}_{t_{n}}^{\pi}|^2dt
\nonumber	\\\leq& C(1+|\xi|^2)h^{1-\epsilon}+ \mathop{\max}\limits_{0\leq n<N}\mathop{\sup}\limits_{t_{n}\leq t \leq t_{n+1}}(\mathbb{E}|\tilde{X}^{\pi}_{t_{n}}-\hat{X}^{\pi}_{t_{n}}|^2+\mathbb{E}|\tilde{Y}^{\pi}_{t_{n}}-\hat{Y}^{\pi}_{t_{n}}|^2)
\nonumber \\&+\sum_{n=0}^{N-1}\int_{t_{n}}^{t_{n+1}}\mathbb{E}|\tilde{Z}^{\pi}_{t}-\hat{Z}_{t_{n}}^{\pi}|^2dt +\sum_{n=0}^{N-1}\int_{t_{n}}^{t_{n+1}}\mathbb{E}|\tilde{\Gamma}^{\pi}_{t}-\hat{\Gamma}_{t_{n}}^{\pi}|^2dt.\nonumber
\end{align}
Choose $\lambda_{1}=1$,  we obtain from (\ref{eq5.10}) and (iii) of Theorem \ref{th5.7},
	\begin{eqnarray*}
	\begin{aligned}	
		\mathbb{E}|\tilde{X}^{\pi}_{t_{n+1}}-\hat{X}^{\pi}_{t_{n+1}}|^2&\leq \mathbb{E}\{(1+Ch)|\tilde{X}^{\pi}_{t_{n}}-\hat{X}^{\pi}_{t_{n}}|^2+Ch|u(t_{n},X^{\pi}_{t_{n}})-\hat{u}(t_{n},X^{\pi}_{t_{n}})|^2\}\\&\leq \mathbb{E}\{(1+Ch)|\tilde{X}^{\pi}_{t_{n}}-\hat{X}^{\pi}_{t_{n}}|^2+C(1+|\xi|^2)h\}.
	\end{aligned}
\end{eqnarray*}
Using Growall inequality and $\tilde{X}^{\pi}_{t_{0}}-\hat{X}^{\pi}_{t_{0}}=0$, we have
  	\begin{eqnarray}\label{eq5.18}
  \sum_{0\leq n\leq N}\mathbb{E}\{|\tilde{X}^{\pi}_{t_{n}}-\hat{X}^{\pi}_{t_{n}}|^2\}\leq C(1+|\xi|^2)h.
  \end{eqnarray}
Choose $\lambda_{2}=\lambda_{3}=\lambda_{4}=\frac{1}{5}$ and $h$ small enough so that $ A_{3}\leq \frac{1}{2}$ and $ A_{4}\leq \frac{1}{2}$. Then we obtain
	\begin{eqnarray*}
	\begin{aligned}			
		\mathbb{E}\left \{|\tilde{Y}^{\pi}_{t_{n}}-\hat{Y}^{\pi}_{t_{n}}|^2\right.&\left.+\frac{1}{2}h|\tilde{Z}^{\pi}_{t_{n}}-\hat{Z}^{\pi}_{t_{n}}|^2
+\frac{1}{2}h|\tilde{\Gamma}^{\pi}_{t_{n}}-\hat{\Gamma}^{\pi}_{t_{n}}|^2\right\}\\&\leq \mathbb{E}\left\{(1+Ch)|\tilde{Y}^{\pi}_{t_{n+1}}-\hat{Y}^{\pi}_{t_{n+1}}|^2+Ch|\tilde{X}^{\pi}_{t_{n}}-\hat{X}^{\pi}_{t_{n}}|^2\right\}.
	\end{aligned}
\end{eqnarray*}
	Since
		\begin{eqnarray*}
		|\tilde{Y}^{\pi}_{t_{N}}-\hat{Y}^{\pi}_{t_{N}}|^2=|g(\tilde{X}^{\pi}_{t_{N}})-g(X^{\pi}_{t_{N}})|^2\leq C|\tilde{X}^{\pi}_{t_{N}}-\hat{X}^{\pi}_{t_{N}}|^2,
	\end{eqnarray*}
we can get
	\begin{eqnarray}\label{eq5.19}
\begin{aligned}
\mathop{\sup }\limits_{0\leq n\leq N}\mathbb{E}\{|\tilde{Y}^{\pi}_{t_{n}}-\hat{Y}^{\pi}_{t_{n}}|^2\}+&h\sum_{n=0}^{N-1}\mathbb{E}|\tilde{Z}^{\pi}_{t_{n}}-\hat{Z}^{\pi}_{t_{n}}|^2
+h\sum_{n=0}^{N-1}\mathbb{E}|\tilde{\Gamma}^{\pi}_{t_{n}}-\hat{\Gamma}^{\pi}_{t_{n}}|^2\\&\leq C \mathop{\sup }\limits_{0\leq n\leq N}\mathbb{E}\{|\tilde{X}^{\pi}_{t_{n}}-\hat{X}^{\pi}_{t_{n}}|^2\\
&\leq C(1+|\xi|^2)h.
\end{aligned}
\end{eqnarray}
Substituting (\ref{eq5.18}) and (\ref{eq5.19}) into (\ref{eq5.17}) yields the desired results. This completes the proof.
\end{proof}

We conclude this section with a remark. Theorem \ref{th5.2} allow us to state that the Euler time discretization scheme achieves a rate of convergence of at least $h^{(1-\epsilon)/2}$ for any $\epsilon>0$ under the standard Lipschitz conditions, and the optimal rate $h^{1/2}$ under the additional assumption.
\section{Error estimates for deep learning approximation} In this section, we derive the simulation error of deep learning method. We first consider algorithm (\ref{eq5.4}) without the terminal condition of $\hat{Y}^\pi_T$. This implies the system has infinitely many solutions, as the case of \eqref{eq5.8}-\eqref{eq5.9}. We have the following estimates whose proof is based on Lemma \ref{le5.4} as well.

 \begin{lemma}\label{le5.8}
 	For sufficiently small $h$, for any $\lambda_{7} >0$ and $\lambda_{8} >\max\{f_{z},\frac{f_{\Gamma}}{K_{\gamma}^2}\}$, let
 	\begin{align*}
 		B_{1}:=&(\lambda_{7}^{-1}+h)b_{y}+\sigma_{y}+\beta_{y},\quad			B_{2}:=-\frac{\ln[1-(1+f_x+\lambda_{8})h]}{ h } ,\\ 			B_{3}:=&\frac{f_{x}}{[1-(1+f_x+\lambda_{8})h]\lambda_{8}}.  	
 		\end{align*}
Suppose equation (\ref{eq5.4}) without the terminal condition of $\hat{Y}^\pi_T$ has two solutions $(\hat{X}^{\pi,j}_{t_{n}}$, $\hat{Y}^{\pi,j}_{t_{n}}$, $\hat{Z}^{\pi,j}_{t_{n}}$, $\hat{\Gamma}^{\pi,j}_{t_{n}})$, $j=1,2$. Let $\Delta X_{t_{n}}=\hat{X}_{t_{n}}^{\pi,1}-\hat{X}_{t_{n}}^{\pi,2}$, $\Delta Y_{t_{n}}=\hat{Y}_{t_{n}}^{\pi,1}-\hat{Y}_{t_{n}}^{\pi,2}$. Then we have
 	\begin{eqnarray}\label{eq3.43}
 \begin{array}{l}	
 \mathbb{E}|\Delta X_{t_{n}}|^2\leq  B_{1} \sum\limits_{i=0}^{n-1} e^{(A_{1}+\lambda_{7})(n-1-i)h}\mathbb{E}|\Delta Y_{t_{i}}|^2h,
 \end{array}
 \end{eqnarray}
 \begin{eqnarray}\label{eq3.44}
 \begin{array}{l}	
 \mathbb{E}|\Delta Y_{t_{n}}|^2\leq e^{B_{2}(N-n)h}|\Delta Y_{t_{N}}|^2+ B_{3} \sum\limits_{i=n}^{N-1}e^{B_{2}(i-n)h}\mathbb{E}|\Delta X_{t_{i}}|^2h.	
 \end{array}
 \end{eqnarray}
 \end{lemma}
 \begin{proof} Similar to (\ref{eq4.13}), it follows that
 	\begin{eqnarray*}
 	\begin{aligned}
 		\mathbb{E}\{|\Delta X_{t_{n+1}}|^2\} \leq [1+(A_{1}+\lambda_{7})]h\mathbb{E}\{|\Delta X_{t_{n}}|^2\}+B_{1}h\mathbb{E}\{|\Delta Y_{t_{n}}|^2\}.
 	\end{aligned}
 \end{eqnarray*}
Using discrete Gr\"onwall inequality yields the estimate (\ref{eq3.43}).

By the martingale representation theorem (see, for example, \cite{Situ05}), there exists an $\mathbb{F}$-adapted square integrable process $\{ \Delta Z_{t} \}_{t_{i}\leq t\leq t_{i+1}}$, $\{\Delta U_{t}(e)\}_{t_{i}\leq t\leq t_{i+1}}$ such that
   	\begin{eqnarray*}
   	\begin{aligned}
   		\Delta Y_{t_{n+1}}=\mathbb{E}_{t_{n}}\{\Delta Y_{t_{n+1}}\}+\int_{t_{n}}^{t_{n+1}}
   		(\Delta Z_{t})^\mathsf{T}dW_{t}+\int_{t_{n}}^{t_{n+1}}
   		\int_{E}\Delta U_{t}(e)\tilde{\mu}(de,(t_{n},t_{n+1}]).
   	\end{aligned}
   \end{eqnarray*}
Then, similar to (\ref{eq5.012}), we have
 	\begin{align}\label{eq5.32}
 		\mathbb{E}\{|\Delta Y_{t_{n+1}}|^2\}
\geq &\mathbb{E}|\Delta Y_{t_{n}}|^2+\int_{t_{n}}^{t_{n+1}}\mathbb{E}|\Delta Z_{t}|^2dt+\int_{t_{n}}^{t_{n+1}}
 		\int_{E}\mathbb{E}|\Delta U(e)|^2\lambda(de)dt
 		\nonumber\\&-{1+f_x}\Delta t_{n}\mathbb{E}\{|\Delta Y_{t_{n}}|^2\}-[\lambda_{8}\mathbb{E}\{|\Delta Y_{t_{n}}|^2\}+\lambda_{8}^{-1}(f_{x}\mathbb{E}\{|\Delta X_{t_{n}}|^2\}\nonumber\\
 &+f_{z}\mathbb{E}\{|\Delta Z_{n}|^2\}+f_{\Gamma}\mathbb{E}\{| \Delta \Gamma_{n}|^2)]\Delta t_{n}\}.
  	\end{align}
Substituting (\ref{eq5.16a}) and
	\begin{eqnarray*}
	\mathbb{E}\left\{\int_{t_{n}}^{t_{n+1}}|\Delta Z_{t}|^2dt\right\} \geq \frac{1}{ \Delta {t_{n}}}\mathbb{E}\left\{\left|\int_{t_{n}}^{t_{n+1}}\Delta Z_{t}dt\right|^2\right\}
\end{eqnarray*}
into (\ref{eq5.32}), we have
 	\begin{eqnarray}\label{eq5.21}
 \begin{aligned}
 \mathbb{E}\{|\Delta Y_{t_{n+1}}|^2\}\geq &[1-(1+f_x+\lambda_{8})\Delta{t_{n}}] \mathbb{E}\{|\Delta Y_{t_{n}}|^2\}+(1-f_{z}\lambda_{8}^{-1})\Delta{t_{n}}\mathbb{E}\{|\Delta Z_{n}dt|^2\}\\&+(\frac{1}{K_\gamma^{2}}-f_{\Gamma}\lambda_{8}^{-1})\Delta{t_{n}}\mathbb{E}\{|\Delta \Gamma_{n}|^2\}-f_{x}\lambda_{8}^{-1}\Delta{t_{n}}\mathbb{E}\{|\Delta X_{t_{n}}|^2\}.
 \end{aligned}
 \end{eqnarray}
 For any $\lambda_{8}>\max\{f_{z},\frac{f_{\Gamma}}{K_{\gamma}^2}\}$, and sufficiently small $h$ satisfying $(2k_{f}+\lambda_{8})h<1$, we then have
 	\begin{eqnarray*}\label{eq5.210}
 	\begin{aligned}
 		\mathbb{E}|\Delta Y_{t_{n}}|^2\leq [1-(1+f_{x}&+\lambda_{8})\Delta{t_{n}}]^{-1} (\mathbb{E}|\Delta Y_{t_{n+1}}|^2+f_{x}\lambda_{8}^{-1}\Delta{t_{n}}\mathbb{E}|\Delta X_{t_{n}}|^2).
 	\end{aligned}
 \end{eqnarray*}
 By induction, we obtain (\ref{eq3.44}) and therefore complete the proof of the theorem. 	
 \end{proof}

 Now we are ready to bound the simulation error of deep learning algorithm.
 \begin{theorem}\label{th5.8}
 Suppose Assumptions 1, 2, 3 and 4, hold true and there exist $\lambda_7>0$ and $\lambda_{8} \ge \max\{f_{z},\frac{f_{\Gamma}}{K_{\gamma}^2}\}$ such that $A_0<1$ with
 $$A_0=(b_{y}\lambda_{7}^{-1}+\sigma_{y}+\beta_{y})\frac{1-e^{-B_{4}T}}{B_{4}}\left[g_{x}(1+\lambda_{9})e^{B_{4}T}
 +f_{x}\lambda_{8}^{-1}\frac{e^{B_{4}T}-1}{B_{4}}\right],$$
  where
	\begin{eqnarray*}
	B_{4}=1+2k_1+\sigma_{x}+\beta_{x}+f_{x}+\lambda_1+\lambda_{7}+\lambda_{8}.
\end{eqnarray*}
If ($X^{\pi}_{t_{n}}$, $Y^{\pi}_{t_{n}}$, $Z^{\pi}_{t_{n}}$, $\Gamma^{\pi}_{t_{n}}$) is a solution of equation \eqref{eq3.16} and ($\hat{X}^{\pi}_{t_{n}}$, $\hat{Y}^{\pi}_{t_{n}}$, $\hat{Z}^{\pi}_{t_{n}}$,$\hat{\Gamma}^{\pi}_{t_{n}}$) is a solution of equation \eqref{eq5.4}, then we have
 	\begin{eqnarray*}
 	\begin{aligned}
 		\mathop{\sup}\limits_{0 \leq n \leq N}&(\mathbb{E}|X^{\pi}_{t_{n}}-\hat{X}_{t_{n}}^{\pi}|^2+\mathbb{E}|Y^{\pi}_{t_{n}}-\hat{Y}_{t_{n}}^{\pi}|^2)\\&+\ \sum_{n=0}^{N-1}\mathbb{E}|Z_{t_{n}}^{\pi}-\hat{Z}_{t_{n}}^{\pi}|^2 \Delta t_{n} +\sum_{n=0}^{N-1}\mathbb{E}|\Gamma_{t_{n}}^{\pi}-\hat{\Gamma}_{t_{n}}^{\pi}|^2\Delta t_{n}           \leq C\mathbb{E}|g(X_{T}^{\pi})-Y_{T}^{\pi}|^2.
 	\end{aligned}
 \end{eqnarray*} 	
 \end{theorem}

\begin{proof}Let
$\hat{X}_{t_{n}}^{\pi,1}=X_{t_{n}}^{\pi}$, $\hat{Y}_{t_{n}}^{\pi,1}=Y_{t_{n}}^{\pi}$, $\hat{Z}_{t_{n}}^{\pi,1}=Z_{t_{n}}^{\pi}$, $\hat{\Gamma}_{t_{n}}^{\pi,1}=\Gamma_{t_{n}}^{\pi}$, $\hat{X}_{t_{n}}^{\pi,2}=\hat{X}_{t_{n}}^{\pi}$, $\hat{Y}_{t_{n}}^{\pi,2}=\hat{Y}_{t_{n}}^{\pi}$, $\hat{Z}_{t_{n}}^{\pi,2}=\hat{Z}_{t_{n}}^{\pi}$ and  $\hat{\Gamma}_{t_{n}}^{\pi,2}=\hat{\Gamma}_{t_{n}}^{\pi}$. Then using Lemma \ref{le5.8} we can bound the difference between ($X^{\pi}_{t_{n}}$, $Y^{\pi}_{t_{n}}$, $Z^{\pi}_{t_{n}}$, $\Gamma^{\pi}_{t_{n}}$) and ($\hat{X}^{\pi}_{t_{n}}$, $\hat{Y}^{\pi}_{t_{n}}$, $\hat{Z}^{\pi}_{t_{n}}$, $\hat{\Gamma}^{\pi}_{t_{n}}$) by the objective function $\mathbb E|g(X_{T}^{\pi})-Y_{T}^{\pi}|^2$.

To begin with, for any $\lambda_{9}>0$, we set
\begin{eqnarray*}\label{eq3.48}
\begin{aligned}
P:=&\mathop{\max}\limits_{n \in [0,N]}e^{-(A_{1}+\lambda_{7})nh}\mathbb{E}|\Delta X_{t_{n}}|^2,\quad S:=\mathop{\max}\limits_{n \in [0,N]}e^{B_{2}nh}\mathbb{E}|\Delta Y_{t_{n}}|^2,\\
A(h):=&B_{1}he^{-(A_{1}+\lambda_{7})h}\frac{e^{-(A_{1}+\lambda_{7}+B_{2})T}-1}{e^{-(A_{1}+\lambda_{7}+B_{2})h}-1}\\& \times\left[g_{x}(1+\lambda_{9})e^{(A_{1}+\lambda_{7}+B_{2})T}+B_{3}h\frac{e^{(A_{1}+\lambda_{7}+B_{2})T}-1}{e^{(A_{1}+\lambda_{7}+B_{2})h}-1}\right].
	\end{aligned}
\end{eqnarray*}
Then when $A(h)<1$, applying Lemma \ref{le5.8} yields
	\begin{eqnarray*}
	\begin{aligned}
		P\leq [1-A(h)]^{-1}e^{B_{2}T}(1+\lambda_{9}^{-1})B_{1}he^{-(A_{1}+\lambda_{7})h} \frac{e^{-(A_{1}+\lambda_{7}+B_{2})T}-1}{e^{-(A_{1}+\lambda_{7}+B_{2})h}-1}\mathbb{E}|g(X_{T}^{\pi})-Y_{T}^{\pi}|^2,
	\end{aligned}
\end{eqnarray*}
and
\begin{eqnarray*}
	\begin{aligned}
		S\leq [1-A(h)]^{-1}e^{B_{2}T}(1+\lambda_{9}^{-1})\mathbb{E}|g(X_{T}^{\pi})-Y_{T}^{\pi}|^2.
	\end{aligned}
\end{eqnarray*}
Noting $\mathop{\lim}\limits_{h \to 0}A(h)=A_0$, we have
\begin{eqnarray*}
	\begin{aligned}
		\mathop{\lim}\limits_{h \to 0}P\leq &2[1-A_0]^{-1}e^{(1+f_{x}+\lambda_{8})T}\\&\times(1+\lambda_{9}^{-1})(b_{y}\lambda_{7}^{-1}+\sigma_{y}+\beta_{y}) \frac{1-e^{-B_{4}T}}{B_{4}}\mathbb{E}|g(X_{T}^{\pi})-Y_{T}^{\pi}|^2
	\end{aligned}
\end{eqnarray*}
and
\begin{eqnarray*}
	\begin{aligned}
		\mathop{\lim}\limits_{h \to 0} S\leq 2[1-A_0]^{-1}e^{(1+f_{x}+\lambda_{8})T}(1+\lambda_{9}^{-1})\mathbb{E}|g(X_{T}^{\pi})-Y_{T}^{\pi}|^2.
	\end{aligned}
\end{eqnarray*}
We then obtain our error estimates of $\mathop{\max}\limits_{n \in [0,N]}\mathbb{E}|\Delta X_{t_{n}}|^2$ and $\mathop{\max}\limits_{n \in [0,N]}\mathbb{E}|\Delta Y_{t_{n}}|^2$ as
\begin{align}\label{eq5.22}
\mathop{\max}\limits_{n \in [0,N]}\mathbb{E}|\Delta X_{t_{n}}|^2\leq C(\lambda_{7},\lambda_{8})\mathbb{E}|g(X_{T}^{\pi})-Y_{T}^{\pi}|^2,\\
\label{eq5.23}
\mathop{\max}\limits_{n \in [0,N]}\mathbb{E}|\Delta Y_{t_{n}}|^2\leq C(\lambda_{7},\lambda_{8})\mathbb E|g(X_{T}^{\pi})-Y_{T}^{\pi}|^2.
\end{align}
To estimate $\mathbb{E}|\Delta Z_{n}|^2$ and $\mathbb{E}|\Delta \Gamma_{n}|^2$, for $\max\{f_{z}, f_{\Gamma}\}\not= 0$, we choose $\lambda_{8}=2\max\{f_{z}, \frac{f_{\Gamma}}{K_{\gamma}^2}\}$ in (\ref{eq5.21}) to get
		\begin{align*}
	\frac{1}{2}\Delta{t_{n}}\mathbb{E}|\Delta Z_{n}|^2+&\frac{1}{2}\Delta{t_{n}}\mathbb{E}|\Delta \Gamma_{n}|^2\\\leq& (1-f_{z}\lambda_{8}^{-1})\Delta{t_{n}}\mathbb{E}|\Delta Z_{t_{n}}|^2+\left(\frac{1}{K_{\gamma}^2}-f_{\Gamma}\lambda_{8}^{-1}\right)\Delta{t_{n}}\mathbb{E}|\Delta \Gamma_{t_{n}}|^2
	\\ \leq& \frac{f_{x}}{2\max\{f_{z},\frac{f_{\Gamma}}{K_{\gamma}^2}\}}\mathbb{E}|\Delta X_{t_{n}}|^2\Delta t_{n}+\mathbb{E}|\Delta Y_{t_{n+1}}|^2\\&-\left[1-\left(1+f_{x}+2\max\left\{f_{z},\frac{f_{\Gamma}}{K_{\gamma}^2}\right\}\right)\Delta{t_{n}}\right] \mathbb{E}|\Delta Y_{t_{n}}|^2,
	\end{align*}
which further implies
	\begin{eqnarray}\label{eq5.24}
\begin{aligned}
\sum_{n=0}^{N-1}\Delta{t_{n}}\mathbb{E}|\Delta Z_{n}|^2+&\sum_{n=0}^{N-1}\Delta{t_{n}}\mathbb{E}|\Delta \Gamma_{n}|^2 \\\leq& \frac{f_{x}T}{\max\{f_{z},f_{\Gamma}\}}\mathop{\max}\limits_{n \in [0,N]}\mathbb{E}|\Delta X_{t_{n}}|^2\\&+[2+(2+2f_x+4\max\{f_{z},f_{\Gamma}\})T\vee 0] \mathop{\max}\limits_{n \in [0,N]}\mathbb{E}|\Delta Y_{t_{n}}|^2
\\\leq& C(\lambda_{7},\lambda_{8})\mathbb{E}|g(X_{T}^{\pi})-Y_{T}^{\pi}|^2.
\end{aligned}
\end{eqnarray}
Note the estimate (\ref{eq5.24}) is trivial for the case of $\max\{f_{z}, f_{\Gamma}\}= 0$. Then combined ({\ref{eq5.22}}), ({\ref{eq5.23}}) and ({\ref{eq5.24}}) leads to the desired results. This completes the proof.
\end{proof}

Finally, combining with Theorems \ref{th5.2} and \ref{th5.8}, we obtain our main Theorem \ref{th5.1}. It is essential to bear in mind that there must exist $\lambda_7$ and $\lambda_8$ satisfying the conditions in Theorem \ref{th5.8}, provided any of the weak coupling and monotonicity conditions introduced in Assumption 3 holds to a sufficient extent.

\section{Numerical Experiments}In this section, we will use two numerical examples to illustrate the effectiveness of the deep learning-based algorithms.

 \subsection{One-dimensional problem}
We first consider a one-dimensional problem (see example 1 of \cite{ZHAO17}). Let $
  g(T,x)=\sin(X_{T}+T)+2$ and
  \begin{eqnarray}\label{eq4.1}
  b(t,X_{t})=0,~~ \sigma(t,X_{t})=1,~~ \beta(X_{t^{-}},e)=e, ~~\delta=1,~~ d=1,\\
  \label{eq4.2}
   \begin{aligned}
f(t,x,u,\sigma^\mathsf{T}\nabla_{x}u,B[u])=&\frac{(u-2)\exp(u)}{2\exp(\sin(x+t)+2)}-\frac{u\nabla_{x}u}{\sin(x+t)+2}\\
&-\int_{E}(u(t,x+e)-u(t,x))\lambda(de),
   \end{aligned}
  \end{eqnarray}
such that the exact solution of  (\ref{eq2.1}) is $u(t,x)=\sin(x+t)+2$. The compensated Poisson random measure:
  \begin{eqnarray*}
    \lambda(de) =\lambda\rho(e)de:=\mathcal{X}_{\left[-\delta,\delta \right]}(e)de,
  \end{eqnarray*}
where $\mathcal{X}_{\left[-\delta,\delta \right]}$ is the characteristic function of the interval $\left[-\delta,\delta \right]$, $\lambda=2\delta$ is the jump intensity and $\rho(e)=\frac{1}{2\delta}\mathcal{X}_{\left[-\delta,\delta \right]}(e)$ is the density function of a uniform distribution on $\left[-\delta,\delta \right]$. Then the FBSDEJ corresponding to (\ref{eq4.1})-(\ref{eq4.2}) is
\begin{eqnarray*}
\left\{
\begin{aligned}
	dX_{t}=& dW_{t}+\int_{E}e\tilde{\mu}(de,dt),\\
    -dY_{t}=&\left\{\frac{(Y_{t}-2)\exp(Y_{t})}{2\exp[\sin(X_{t}+t)+2]}-\frac{Y_{t}Z_{t}}{\sin(X_{t}+t)+2}
	-\Gamma_{t}\right\}dt\\&- Z_{t}dW_{t}-\int_{E}U_{t}(e)\tilde{\mu}(de,dt).
\end{aligned}
\right.
\end{eqnarray*}

Now, let us set $N=20$ and set $2$ hidden layers, both of which are $1+10$ dimensional. Input layer and output layer are chosen as $1$-dimensional. Table 4.1 depicts average value of $u(0,X_{0})$ and standard deviation of $u(0,X_{0})$ based on $256$ Monte Carlo samples and $5$ independent runs. From Table 4.1, we observe that we can obtain a good approximation of $u(0,X_{0})$ by using the deep learning-based algorithm.
\begin{table}[H]
	\centering
	\caption{ Estimate of $u(0,X_{0})$ where $X_{0}=0$, $d=1$.}
	\begin{tabularx}{22em}
		{|*{4}{>{\centering\arraybackslash}X|}}
		\hline
		 & Averaged value & Standard deviation & Loss function \\\hline
		   0 & 1.63119     & 0.18441            & 0.81499 \\\hline
		1000 & 1.91521     & 0.04239            & 0.17286 \\\hline
		2000 & 1.96693     & 0.01654            & 0.14002\\\hline
		3000 & 1.98162     & 0.00939            & 0.13258\\\hline
		4000 & 1.99324     & 0.00639            & 0.12275 \\\hline
	\end{tabularx}

\end{table}
To further demonstrate the effectiveness of this algorithm for decoupled FBSDEJs, the relative $L^{1}$-approximation error of  $u(0,X_{0})$ and mean of the loss function are presented in Fig. 4.1. It is observed from Fig. 4.1 that the relative $L^{1}$-approximation error of $u(0,X_{0})$ and mean of the loss function drop significantly as the number of iteration steps increase from $0$ to $1500$, but is extremely slow as the number of iteration steps increase from $1500$ to $4000$.
\begin{figure}[H]
		\centering
		\subfigure[Relative $L^{1}$-approximation error]{\includegraphics[scale=0.41]{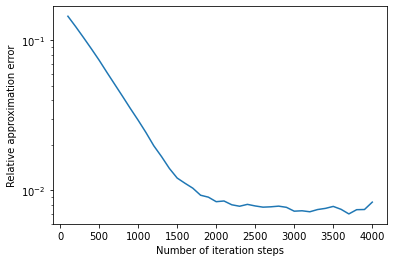} \label{1}}
		\quad
		\subfigure[Mean of the loss function]{\includegraphics[scale=0.42]{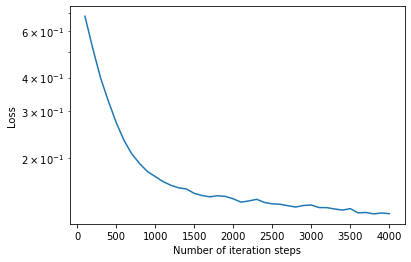} \label{2}}
		\caption{Relative  $L^{1}$-approximation error of  $u(0,X_{0})$ and mean of the loss function.}
		
	\end{figure}


 \subsection{High-dimensional problems} In the second example, we consider a high-dimensional problems with
\begin{eqnarray}\label{eq4.3}
b(t,X_{t})=0,~~\sigma(t,X_{t})=\frac{1}{\sqrt
	d}I_{d},~~\beta(X_{t^{-}},e)=e,~~\delta=1,~~d=100,\\
\begin{aligned}\label{eq4.4}
f(t,x,u,\sigma^\mathsf{T}\nabla_{x}u,B[u])=&\frac{(u-2)\exp(u)}{2\exp(\sin(\bar{x}+t)+2)}-\frac{u(I_{d}  \nabla_{x}u)}{\sin(\bar{x}+t)+2}\\
&-\int_{E}(u(t,\bar{x}+e)-u(t,\bar{x}))\lambda(de),
\end{aligned}
\end{eqnarray}
where $\bar{x}=\sum_{i=1}^{d}x_i$. We choose $g(T,x)=\sin(X_{T}+T)+2$ such that the exact solution of the associated PIDEs is $ u(t,x)=\sin(\bar{x}+t)+2 $. The compensated Poisson random measure is the same as in the first example:
\begin{eqnarray*}
\lambda(de) =\lambda\rho(e)de:=\mathcal{X}_{\left[-\delta,\delta \right]}(e)de.
\end{eqnarray*}
Similarly, we can obtain the corresponding FBSDEJs to (\ref{eq4.3})-(\ref{eq4.4}).
Let us still set $N=20$ and $2$ hidden layers. Both of hidden layers are $100+10$ dimensional, input layer is $100$-dimensional, and output layer is $1$-dimensional.  Table 4.2 depicts average value of $u(0,X_{0})$ and standard deviation of $u(0,X_{0})$ based on $256$ Monte Carlo samples and $5$ independent runs. From Table 4.2, we still observe that the deep learning-based algorithm can produce a good approximation of $u(0,X_{0})$.
\begin{table}[H]
	\centering
		\caption{Estimate of $u(0,X_{0})$ where $X_{0}=0$, $d=100$.}
	\begin{tabularx}{22em}
		{|*{4}{>{\centering\arraybackslash}X|}}
		\hline
		& Averaged value & Standard deviation & Loss function \\\hline
		   0 & 1.96204     & 0.01881            & 0.50262 \\\hline
		 500 & 1.99250     & 0.00716            & 0.50266 \\\hline
		1000 & 1.99322     & 0.00683            & 0.50098\\\hline
		1500 & 1.99942     & 0.00602            & 0.50026\\\hline
		2000 & 2.00047     & 0.00714            & 0.49976 \\\hline
	\end{tabularx}

\end{table}
The relative $L^{1}$-approximation error of  $u(0,X_{0})$ and mean of the loss function are presented in Fig. 4.2 from which we observe that the relative $L^{1}$-approximation error of $u(0,X_{0})$ oscillates when the number of iteration steps becomes larger. We also see that the mean of the loss function decays as the number of iteration steps increase.
\begin{figure}[H]
	\centering
	\subfigure[Relative $L^{1}$-approximation error]{\includegraphics[scale=0.42]{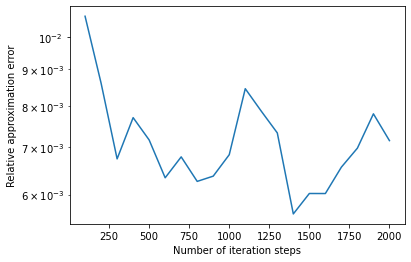} \label{3}}
	\quad
	\subfigure[Mean of the loss function]{\includegraphics[scale=0.41]{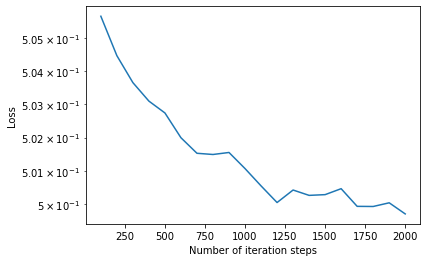} \label{4}}
		\caption{Relative  $L^{1}$-approximation error of  $u(0,X_{0})$ and mean of the loss function.}
	
\end{figure}

\section{Concluding remarks}
In this work, we popularized the deep BSDE schemes for high-dimensional forward-backward stochastic differential equations with jumps (FBSDEJs) and related high-dimensional parabolic integral-partial differential equations (PIDEs). We constructed the deep FBSDE scheme in which deep neural networks are used to approximate the gradient and the integral kernel. Then the error estimates for this deep FBSDE algorithm were obtained based on the optimal error estimates of Euler time discretization and deep learning error estimate which is bounded by the objective function in the variational problems. To establish the convergence relationship between the solutions of FBSDEJs and PIDEs, the Markovian iteration has been introduced and its convergence was established. We have implemented this deep FBSDE scheme for low and high dimensional FBSDEJs problems and numerical results showed that this scheme is effective. We also realized that our scheme cannot reach the accuracy of the classical numerical schemes which are not available for high dimensional problems. To improve the accuracy, extending our scheme to DBDP2 in which the loss function will be minimized on each time step will be our future work.


\begin{thebibliography}{10}
  \bibitem{Achdou05} {\sc Y. Achdou and O. Pironneau}, {\em Computational methods for option pricing}, {\rm Frontiers Appl. Math.}, {\bf 30}, SIAM, Philadelphia, PA, 2005.

  \bibitem{BBP1997} {\sc G. Barles, R. Buckdahn, and E. Pardoux}, {\em Backward stochastic differential equations and integral-partial differential equations}, Stoch. Stoch. Rep., 60(1997), pp. 57–83.

  \bibitem{B2006} {\sc D. Becherer}, {\em Bounded solutions to backward SDEs with jumps for utility optimization and indifference hedging}, Ann. Appl. Probab., 16(2006), pp. 2027–2054.

  \bibitem{Beck19} {\sc C. Beck, W. E, and A. Jentzen}, {\em Machine learning approximation algorithms for high-dimensional fully nonlinear partial differential equations and second-order backward stochastic differential equations}, J. Nonlinear Sci., 29 (2019), pp. 1563-1619.

  \bibitem{Bellman1957}{\sc R. Bellman}, {\em Dynamic programming}, Princeton Landmarks in Mathematics. Princeton University Press, Princeton, Nj, (2010).

  \bibitem{BZ2008} {\sc C. Bender and J. Zhang}, {\em Time discretization and Markovian iteration for coupled FBSDEs}, Ann. Appl. Probab., 18(2008), pp. 143-177.

  \bibitem{E2007}{\sc B. Bouchard and R. Elie}, {\em Discrete-time approximation of decoupled forward–backward SDE with jumps}, Stochastic Process. Appl., 118(2008), pp. 53–75.

  \bibitem{Buckdahn94}{\sc B. Buchdahn and E. Pardoux}, {\em BSDE's with jumps and associated integro-partial differential equations},SFB 373 Discussion Papers 1994, 41, Humboldt University of Berlin, Interdisciplinary Research Project 373: Quantification and Simulation of Economic Processes..

  \bibitem{Castro21} {\sc J. Castro}, {\em Deep learing schemes for parabolic nonlocal integro-differetial equations}, arXiv preprint arXiv: 2103.15008v1 (2021).

  \bibitem{Chan19} {\sc Q. Chan-Wai-Nam, J. Mikael, and X. Warin}, {\em Machine learing for semi-linear PDEs}, J. Sci. Comput., 79 (2019), pp. 1667-1712.

  \bibitem{Cont03} {\sc R. Cont and P. Tankov}, {\em Financial modelling with jump processes}, Chanpman and Hall/CRC Press, London, 2004.



    \bibitem{WA17}{\sc W. E, J. Han, and A. Jentzen}, {\em Deep learning-based numerical methods for high-dimensional parabolic partial differential equations and backward stochastic differential equations}, Commun. Math. Stat., 5(2017), pp. 349–380.

 \bibitem{WMAT19}{\sc W. E, M. Hutzenthaler, A. Jentzen, and T. Kruse}, {\em On multilevel Picard numerical approximations for high-dimensional nonlinear parabolic partial differential equations and high-dimensional nonlinear backward stochastic differential equations}, J. Sci. Comput., 79(2019), pp. 1534–1571.

  \bibitem{KPQ2014} {\sc N. El Karoui, S. Peng, and M. C. Quenez}, {\em Backward stochastic differential equations in finance}, Math. Finance., 7(1997), pp. 1–71.

  \bibitem{Fu2016}{\sc Y. Fu, J. Yang, and W. Zhao}, {\em Prediction-Correction scheme for decoupled forward backward stochastic differential equations with jumps}, East Asian J. Appl. Math., 6(2016), pp. 253-277.

  \bibitem{Gobet05} {\sc E. Gobet, J.-P. Lemor, and X. Warin}, {\em A regression-based Monte Carlo method to solve backward stochastic differential equations}, Ann. Appl. Problb., 15 (2005), pp. 2172-2202.

  \bibitem{Gonon21a} {\sc L. Gonon and C. Schwab}, {\em Deep ReLU network ecpression rates for option prices in high-dimensional, exponential L\'evy models}, arXiv preprint arXiv: 2101.11897v2, (2021).

  \bibitem{Gonon21b} {\sc L. Gonon and C. Schwab}, {\em Deep ReLU neural network overcome the curse of dimensionality for partial integrodifferential equations}, arXiv preprint arXiv: 2102.11707v2, (2021).

\bibitem{WEM18} {\sc J. Han, J. Arnulf, W. E}, {\em Solving high-dimensional partial differential equations using deep learning}, Proc. Natl. Acad. Sci. USA, 115(2018), pp. 8505-8510.

  \bibitem{HJ2020}{\sc J. Han and J. Long}, {\em Convergence of the deep BSDE method for coupled FBSDEs}, Probab. Uncertain. Quant. Risk, 5(2020), pp. 1-33.

  \bibitem{HL2012}{\sc P. Henry-Labordere}, {\em Counterparty risk valuation: A marked branching diffusion approach}, arXiv preprint arXiv:1203.2369, (2012).

  \bibitem{PNT2016}{\sc P. Henry-Labordere, N. Oudjane, X. Tan, N. Touzi, and X. Warin}, {\em Branching diffusion representation of semilinear PDEs and Monte Carlo approximation}, Ann. Inst Henri Poincaré Probab. Stat., 55(2019), pp. 184–210.

  \bibitem{HPH2020}{\sc C. Hur\'e, H. Pham, and X. Warin}, {\em Deep backward schemes for high-dimensional nonlinear PDEs}, Math. Comp., 89(2020), pp. 1547–1580.

  \bibitem{MATTP18}{\sc M. Hutzenthaler, A. Jentzen, T. Kruse, T. A. Nguyen, and P. von Wurstemberger}, {\em  Overcoming the curse of dimensionality in the numerical approximation of semilinear parabolic partial differential equations}, Proc. R. Soc. A., 476 (2020).

  \bibitem{Hutzenthler20}{\sc M. Hutzenthaler and T. Kruse}, {\em Multi-level Picard approximations of high-dimensional semilinear parabolic differential equations with gradient-dependent nonlinearities}, SIAM J. Numer. Anal., 58 (2020), pp. 929-961.

  \bibitem{Kadalbajoo17} {\sc M. K. Kadalbajoo, L. P.Tripathi and A. Kumar}, {\em An error analysis of a finite element method with IMEX-time semidiscretizations for some partial integro-differential inequalities arising in the pricing of American options}. SIAM J. Numer. Anal., 55 (2017), pp. 869-891.

  \bibitem{Ma07} {\sc J. Ma and J. Yong}, {\em forward-backward stochastic differential equations and their applicatiions}, Springer, Berlin Heidelberg, 2007.

  \bibitem{Mao22} {\sc M. L. Mao, W. S. Wang and X. Jiang}, {\em An extrapolated Crank-Nicolson method for option pricing under stochastic volatility model with jump}, Submitted.

  \bibitem{PP1990} {\sc E. Pardoux and S. Peng}, {\em Adapted solution of a backward stochastic differential equation}, Systems Control Lett., 14(1990), pp. 55-61.

  \bibitem{P1993} {\sc S. Peng}, {\em Backward stochastic differential equations and applications to optimal control}, Appl. Math. Optim., 27(1993), pp. 125-144.

  \bibitem{Pindza14} {\sc E. Pindza, K. C. Patidar and E. Ngounda}, {\em Robust spectral method for numerical valuation of European options under Merton's jump-diffusion model}, Numer. Methods Partial Differential Equations, 30 (2014), pp. 1169-1188.

  \bibitem{JK2018}{\sc J. Sirignano and K. Spiliopoulos}, {\em DGM: A deep learning algorithm for solving partial differential equations}, J. Comput. Phys., 375(2018), pp. 1339–1364.

  \bibitem{R1997}{\sc	R. Situ}, {\em On solutions of backward stochastic differential equations with jumps and applications}, Stochastic Process. Appl., 66(1997), pp. 209–236.

  \bibitem{Situ05} {\sc R. Situ}, {\em Theory of stochastic differential equations with jumps and applications}, Springer, Berlin, 2005.

  \bibitem{TL1994}{\sc S. Tang and X. Li}, {\em Necessary conditions for optimal control of stochastic systems with random jumps}, SIAM J. Control Optim., 32(1994), pp. 1447-1475.

  \bibitem{Wang19} {\sc W. S. Wang, Y. Z. Chen and H. Fang}, {\em On the variable two-step IMEX BDF method for parabolic integro-differential equations with nonsmooth initial data arising in finance}, SIAM J. Numer. Anal., 57 (2019), pp. 1289-1317.

  \bibitem{WMZ2021}{\sc W. Wang, M. Mao and Z. Wang}, {\em An efficient variable step-size method for options pricing under jump-diffusion models with nonsmooth payoff function}, ESAIM Math. Model. Numer. Anal., 55(2021),  pp. 913-938.

  \bibitem{Warina18} {\sc X. Warin}, {\em Nesting Monte Carlo for high-dimensional non-linear PDEs}, Monte Carlo Methods Appl., 24(2018), pp. 225-247.

  \bibitem{W2003}{\sc Z. Wu}, {\em Fully coupled FBSDE with Brownian motion and Poisson process in stopping time duration}, J. Aust. Math. Soc., 74(2003), pp. 249–266.
	
  \bibitem{ZHAO17}{\sc W. Zhao, Z. Wei, and G. Zhang}, {\em Second-order numerical schemes for decoupled forward-backward stochastic  differential equations with jumps}, J. Comput. Math., 35(2017), pp. 213-244.

  \bibitem{ZHAO16}{\sc W. Zhao, Fu. Y, and T. Zhou}, {\em Multistep schemes for forward backward stochastic differential equations with jumps}, J. Sci. Comput., 69(2016), pp. 1-22.
	
  \bibitem{Z1999}{\sc W. Zhen}, {\em Forward-backward stochastic differential equations with Brownian motion and Poisson process}, Acta Math. Appl. Sin., 15(1999), pp. 433–443.

\end{thebibliography}
\end{document}